\documentclass[12pt]{article}
\usepackage{graphicx}
\usepackage{amsmath,amsthm,amssymb,enumerate,enumitem}

\usepackage{euscript,mathrsfs}
\usepackage{color}
\usepackage{dsfont}
\usepackage[citecolor=blue,colorlinks=true]{hyperref}
\usepackage[left=2cm,right=2cm,top=2.5cm,bottom=2.5cm]{geometry}
\usepackage{color}
\usepackage[framemethod=tikz]{mdframed}
\allowdisplaybreaks

\usepackage{bm}
\usepackage{subcaption}

\usepackage{soul}
\usepackage[normalem]{ulem}

\catcode`\@=11 \@addtoreset{equation}{section}

\catcode`\@=12
\usepackage{cancel}

\allowdisplaybreaks
\newtheorem{Theorem}{Theorem}[section]
\newtheorem{Proposition}[Theorem]{Proposition}
\newtheorem{Lemma}[Theorem]{Lemma}

\theoremstyle{definition}

\newtheorem{Remark}[Theorem]{Remark}

\newcommand{\prst}{\mathcal{P}}

\newcommand{\Vrmh}{\{ \vrh, \vuh,\vth \} }

\newcommand{\jump}[1]{[\! [ #1 ] \! ]}

\newcommand{\Td}{\mathbb{T}^d}

\newcommand{\vrh}{\vr_h}
\newcommand{\vrhk}{\vr_{h_k}}

\newcommand{\vthk}{\vt_{h_k}}

\newcommand{\bfphi}{\boldsymbol{\varphi}}

\newcommand{\ds}{\,\mathrm{d}S_x}

\newcommand{\vth}{\vt_h}

\newcommand{\facesint}{\mathcal{E}}
\newcommand{\grid}{\mathcal{T}}

\newcommand{\vuh}{\vu_h}
\newcommand{\vvh}{\vv_h}

\newcommand{\TS}{\Delta t}

\newcommand{\Divh}{{\rm div}_h}
\newcommand{\Gradh}{\nabla_h}

\newcommand{\Ov}[1]{\overline{#1}}

\newcommand{\aleq}{\stackrel{<}{\sim}}

\newcommand{\avs}[1]{\left\{\!\!\left\{ #1\right\}\!\!\right\}}
\newcommand{\vQh}{\vc{Q}_h}
\newcommand{\pd}{\partial}

\newcommand{\Up}{{\rm Up}}
\newcommand{\Fup}{\mbox{F}_h^\eps}
\newcommand{\eps}{\varepsilon}
\newcommand{\muh}{h^\eps}
\newcommand{\abs}[1]{\left| #1\right|}
\newcommand{\norm}[1]{\left\lVert#1\right\rVert}

\newcommand{\bS}{{\mathbb S}}
\newcommand{\Dhuh}{{\mathbb D}_h \vuh}

\newcommand{\Un}[1]{\underline{#1}}
\newcommand{\vr}{\varrho}

\newcommand{\tvr}{\tilde \vr}
\newcommand{\tvu}{{\tilde \vu}}
\newcommand{\tvt}{\tilde \vt}

\newcommand{\vt}{\vartheta}
\newcommand{\vu}{\vc{u}}
\newcommand{\vm}{\vc{m}}

\newcommand{\vn}{\vc{n}}

\newcommand{\vc}[1]{{\bf #1}}
\renewcommand{\vc}[1]{{\bm #1}}

\newcommand{\Div}{{\rm div}_x}
\newcommand{\Grad}{\nabla_x}

\newcommand{\dx}{\,{\rm d} {x}}

\newcommand{\dt}{\,{\rm d} t }

\newcommand{\vU}{\vc{U}}

\newcommand{\dxdt}{\dx  {\rm d} {t}}

\newcommand{\intfacesint}[1]{\int_{\facesint}  #1 \ds}

\newcommand{\intTd}[1]{\int_{\Td} #1 \dx}
\newcommand{\intTdB}[1]{\int_{\Td} \left(#1\right) {\rm d} {x}}

\newcommand{\intTauTd}[1]{\int_0^{\tau} \int_{\Td} #1  \dxdt}
\newcommand{\intTauTdB}[1]{\int_0^{\tau} \int_{\Td} \left(#1  \right)\dxdt}

\newcommand{\vv}{\vc{v}}

\newcommand{\ep}{\epsilon}
\newcommand{\R}{\mathbb{R}}

\newcommand{\I}{\mathbb{I}}

\newcommand{\expe}[1]{ \mathbb{E} \left[ #1 \right] }
\newcommand{\Expec}{ \expe  }
\newcommand{\Dev}[1]{ \mbox{Dev} \left[ #1 \right] }

\newcommand{\br}{ \nonumber \\ }

\def\softd{{\leavevmode\setbox1=\hbox{d}%
          \hbox to 1.05\wd1{d\kern-0.4ex{\char039}\hss}}}
\definecolor{Cgrey}{rgb}{0.85,0.85,0.85}
\definecolor{Cblue}{rgb}{0.50,0.85,0.85}
\definecolor{Cred}{rgb}{1,0,0}
\definecolor{fancy}{rgb}{0.10,0.85,0.10}

\mdfdefinestyle{MyFrame}{%
	linecolor=black,
	outerlinewidth=1pt,
	roundcorner=5pt,
	innertopmargin=\baselineskip,
	innerbottommargin=\baselineskip,
	innerrightmargin=10pt,
	innerleftmargin=10pt,
	backgroundcolor=white!20!white}

\date{}

\newcommand{\RE}[2]{R_E\left((#1)\mid(#2)\right)}
\newcommand{\EH}[2]{\mathbb{E}_{\cal H}\big((#1)|(#2)  \big)}
\newcommand{\Hc}{\mathcal{H}_{\vT}}
\newcommand{\vT}{\tvt}

\usepackage{cancel}

\begin{document}


\title{Convergence analysis of the Monte Carlo method for random Navier--Stokes--Fourier system}

\author{M\' aria Luk\' a\v cov\' a -- Medvi\softd ov\' a\thanks{
\hspace*{1em}The work of M.L. was supported by the Deutsche Forschungsgemeinschaft (DFG, German Research
Foundation) - Project number 233630050 - TRR 146 as well as by TRR 165 Waves to Weather.
M.L. is grateful to the Gutenberg Research College and Mainz Institute of Multiscale Modelling for supporting her research. 
\newline \hspace*{1.2em}$^\spadesuit$ \hspace*{0.3em}The work of B.S. was  supported by the National Natural Science Foundation of China under grant No. 12201437.
}
\and Bangwei She$^{\spadesuit}$
\and Yuhuan Yuan$^{\dagger,*}$
}

\date{}

\maketitle


\centerline{$^*$Institute of Mathematics, Johannes Gutenberg-University Mainz}
\centerline{Staudingerweg 9, 55 128 Mainz, Germany}
\centerline{lukacova@uni-mainz.de}

\medskip
\centerline{$^\spadesuit$Academy for Multidisciplinary studies, Capital Normal University}
\centerline{West 3rd Ring North Road 105, 100048 Beijing, P. R. China}
\centerline{bangweishe@cnu.edu.cn}

\medskip
\centerline{$^\dagger$School of Mathematics, Nanjing University of Aeronautics and Astronautics}
\centerline{Jiangjun Avenue No. 29, 211106 Nanjing, P. R. China}
\centerline{yuhuanyuan@nuaa.edu.cn}



\begin{abstract}
In the present paper we consider the initial data, external force, viscosity coefficients, and heat conductivity coefficient as random data for the compressible Navier--Stokes--Fourier system. The Monte Carlo method, which is frequently used for the approximation of statistical moments, is combined with a suitable deterministic discretisation method in physical space and time. Under the assumption that numerical
densities and temperatures are bounded in probability, we prove the convergence of random finite volume solutions to a statistical strong solution by applying  genuine stochastic compactness arguments. Further, we show the convergence and error estimates for the Monte Carlo estimators of the expectation and deviation. We present several numerical results to illustrate the theoretical results.

\end{abstract}





\section{Introduction}

Randomness is an inherent property of models in science and engineering. Model parameters as well as initial and boundary data are typically known only from observations or measurements that can be abounded by several errors. In order to propagate data uncertainty in the solution of an underlying model, different  methods have been developed in the recent years. The Monte Carlo method that is based on statistical sampling is probably the most popular among them. Although it suffers from relatively slow convergence rate with respect to the ensemble size, its advantage is that it does not suffer from the curse of data dimensionality. The latter is a typical property of spectral/pseudo-spectral or other discretisation methods, see, e.g., \cite{AM17, Bijl, Xiu} and the references therein.

The aim of the present paper is to rigorously analyse the Monte Carlo method for heat conductive, viscous compressible fluid flows subjected to random data.
We recall the Navier--Stokes--Fourier system governing the motion of such fluid flows
\begin{mdframed}[style=MyFrame]
\vspace{-0.5cm}
\begin{align}
\partial_t \vr + \Div (\vr \vu)& = 0, \label{i1}	\\
\partial_t (\vr \vu) + \Div (\vr \vu \otimes \vu) + \Grad p &= \Div \mathbb{S}(\mu, \lambda, \Grad \vu) + \vr \vc{g}, \label{i2} \\
c_v (\partial_t (\vr \vt) + \Div (\vr \vu  \vt) ) - \kappa \Delta \vt   & =  \mathbb{S}(\mu, \lambda, \Grad \vu):\Grad \vu + p \Div \vu , \label{i3}	\\
c_v = 1/(\gamma-1), \ \gamma > 1, \ \mathbb{S} (\mu, \lambda, \Grad \vu) &= \mu \Big( \Grad \vu + \Grad^t \vu - \frac{2}{d} \Div \vu \mathbb{I} \Big) + \lambda \Div \vu \mathbb{I}. \nonumber
\end{align}

\noindent
{\bf Boundary and initial conditions}
\begin{equation} \label{i4}
	x \in \Td \equiv \left( [-1,1]|_{\{ -1, 1\} } \right)^d,\ d=2,3,
\end{equation}
%
\begin{equation} \label{i5}
	\vr(0, \cdot) = \vr_0, \  \vu (0, \cdot) = \vu_0, \  \vt (0, \cdot) = \vt_0.
\end{equation}

\end{mdframed}
Here $\vr, \vu$ and  $\vt$ are the fluid density, velocity, and absolute temperature, respectively. For pressure $p$, we assume the perfect gas law, i.e.~$p=\vr\,\vt.$ Further, $\gamma$ is the adiabatic coefficient, $c_v$ is the specific heat per constant volume.  In what follows we consider the following model data
\begin{mdframed}[style=MyFrame]


\noindent driving force \dotfill $\vc{g} = \vc{g}(x)$;

\noindent viscosity coefficients \dotfill $\mu > 0$, $\lambda \geq 0$;

\noindent heat conductivity coefficient \dotfill $\kappa > 0$;

\noindent initial data \dotfill $\vr_0$, $\vu_0$, $\vt_0$.
\end{mdframed}

\subsection{Data dependence of solution}
In our analysis it is important to specify data dependence of a solution of the Navier--Stokes--Fourier system. Let us
denote by $\vm$ the momentum, by  $E$ the energy and by $S$ the (physical) entropy
\begin{align*}
\vm = \vr \vu, \quad E=\frac12 \vr |\vu|^2 +c_v \vr \vt,  \quad S = \vr s, \quad
s = \log \left( \frac{\vt^{c_v}}{\vr}\right).
\end{align*}
Then the strong solution $(\vr, \vu, \vt)$ to the problem \eqref{i1}--\eqref{i5} satisfies the mass conservation, the energy and entropy balances
\begin{align} \label{mass-conservation}
& \intTd{\vr(\tau, \cdot)} =  \intTd{\vr(0,\cdot)} =: M_0,
\\ \label{energy-balance}
& \intTd{E(\tau, \cdot)} =  \intTd{E(0,\cdot)} + \intTauTd{\vr \vu \cdot \vc{g}},
\\ \label{entropy-inequality}
&  \intTd{ S (\tau, \cdot)} = \intTd{ S (0, \cdot)} + \intTauTdB{\frac{\bS:\Grad \vu}{\vt} +\kappa \frac{|\Grad \vt|^2}{\vt^2}}
\end{align}
for any $\tau \in [0,T].$

Let us introduce the
relative energy functional
\begin{align*}
&\RE{\vr,\vu,\vt}{\tvr,\tvu,\tvt}
=  \frac12 \vr |\vu - \tvu|^2 + \EH{\vr,\vt}{\tvr,\tvt}
\br
&  \mbox{with } \EH{\vr,\vt}{\tvr,\tvt} = \Hc(\vr,\vt)  - \frac{\pd \Hc(\tvr,\tvt)}{\pd \vr}(\vr -\tvr) - \Hc(\tvr,\tvt)
\br
& \mbox{and } \Hc(\vr,\vt) = \vr\left( c_v   \vt  -   \vT s(\vr,\vt)   \right),
\end{align*}
where $(\tvr,\tvu,\tvt)$ is an arbitrary smooth function satisfying $\tvr > 0,\, \tvt > 0$.
As shown, e.g., in \cite{BLMSY}, $R_E$ is a convex function of $(\vr, \vm, S)$ and it holds  
\begin{equation*}
R_E \geq \mathbb{E}_{\mathcal{H}} \geq 0. 
\end{equation*}
Choosing
$
(\tvr,\tvu,\tvt) = (1,0,1)
$
and assuming $\norm{\vc{g}}_{C^1( \Td)} \leq \Ov{g}$ we obtain
\begin{align*}
& \frac{\pd \Hc(\vr,\vt)}{\pd \vr} = c_v \vt - \vT \big( s  -1 \big), \\
& \RE{\vr,\vu,\vt}{1,0,1} = \frac12 \vr |\vu|^2 + \EH{\vr,\vt}{1,1}  = E - S - (c_v + 1) \vr +1
\end{align*}
and
\begin{align*}
\intTauTd{\vr \vu \cdot \vc{g}} & \leq \intTauTd{\frac12 \vr |\vu|^2} \dt + \frac12 |\Ov{g}|^2 \intTauTd{ \vr }
\br
&\leq \intTauTd{\RE{\vr,\vu,\vt}{1,0,1} } + \frac12 |\Ov{g}|^2 T M_0.
\end{align*}
Combining \eqref{mass-conservation} -- \eqref{entropy-inequality} we get the relative energy inequality
\begin{align*}
&\intTd{\RE{\vr,\vu,\vt}{1,0,1}(\tau, \cdot)}  \leq \intTd{\RE{\vr_0,\vu_0,\vt_0}{1,0,1}}  \br
&\hspace{2cm}+ \intTauTd{\RE{\vr,\vu,\vt}{1,0,1}(t,\cdot) } + \frac12 |\Ov{g}|^2 T M_0.
\end{align*}
Applying Gronwall's inequality we have for any $t \in [0,T]$ that
\begin{align*}
0 & \leq \intTd{\RE{\vr,\vu,\vt}{1,0,1} (t,\cdot)}  \br
& \leq \frac12 |\Ov{g}|^2 T M_0 + e^{T}  \intTd{\RE{\vr_0,\vu_0,\vt_0}{1,0,1}}.
\end{align*}
Taking into account the convexity of the relative energy with respect to $(\vr, \vm, S)$  it is easy to check by direct calculation 
that for any  $t\in[0,T]$ 
\begin{align*}
\norm{\vr(t, \cdot)}_{L^1(\Td)} + \norm{\vm(t, \cdot)}_{L^1(\Td)} + \norm{S(t, \cdot)}_{L^1(\Td)} \aleq 1+\intTd{\RE{\vr,\vu,\vt}{1,0,1}(t,\cdot)}. 
\end{align*}
Here and hereafter we write $a\aleq b$ if $a \leq C \, b$ with a positive constant $C$.
Assuming the boundedness of the initial relative energy we obtain
\begin{align}\label{space-1}
&\norm{\vr}_{L^{\infty}(0,T; L^1(\Td))} + \norm{\vm}_{L^{\infty}(0,T; L^1(\Td))} + \norm{S}_{L^{\infty}(0,T; L^1(\Td))} \aleq 1+ \intTd{\RE{\vr_0,\vu_0,\vt_0}{1,0,1}}
\br
&\hspace{1cm}\leq C( \Ov{g},T) \left( 1 + \intTdB{E(\vr_0, \vu_0, \vt_0) - S(\vr_0, \vt_0) - (c_v+1)\vr_0 + 1}\right).
\end{align}

The above  estimate shows that $L^\infty(0,T; L^1(\Td))$ norm of a strong solution of the Navier--Stokes--Fourier system\eqref{i1}-\eqref{i5} is bounded by the data.

\medskip
In this paper we will analyse the Navier--Stokes--Fourier system subject to \emph{random model data} specified above.
This will be done by applying the Monte Carlo method combined with a finite volume (FV) method for space-time discretisation. Our goal is to derive rigorous convergence and error analysis both with respect to statistical sampling as well as space-time discretisation. Although the Monte Carlo approximations, such as Monte Carlo FV methods, are routinely used for uncertainty quantification in computation fluid dynamics or in meteorology, their rigorous convergence and error analysis for compressible viscous and heat conducting fluid flows is still missing in the literature. This paper presents the first results in this direction.

We  refer to our recent work \cite{FLSY_MC1}, where convergence and error estimates of a Monte Carlo FV method for the random compressible barotropic Navier--Stokes system were analysed.  In contrary to the viscous barotropic case, the analysis of heat conductive viscous compressible fluid flow is more involved.  First, the existence of global weak solution for the Navier--Stokes--Fourier equations with perfect gas law $p=  \varrho \vartheta$ is an open problem. There are only some results on the global-in-time existence of weak solutions available, however certain structural restrictions on $p,e,s$ and the coefficients $\mu, \kappa$ are required, see \cite[Theorem 3.1]{FeiNov_book}.

One of the main tool in the  convergence analysis of deterministic discretisation methods, e.g.~FV methods, is the so-called  \emph{weak-strong uniqueness principle} \cite{FeLMMiSh}. This means that a generalised solution (dissipative weak solution), that is identified as a limit of a sequence of discrete solutions, coincides with the strong solution as long as the latter exists.
Second problem lies in the fact that the (dissipative) weak-strong uniqueness results are only conditional for the Navier--Stokes--Fourier system, see~\cite{brezina}. For example, boundedness of density and temperature has to be assumed for the weak-strong uniqueness principle.

Taking these facts into account, we will work in the framework of strong statistical solutions and take boundedness of numerical densities and temperatures in probability as our \emph{principle hypothesis}.
In other words, statistically significant solutions
of the Navier--Stokes--Fourier system do not blow up in density and temperature.
As already suggested by Nash in his pioneering work~\cite{Nash} such a hypothesis is very natural.
We refer to  the recent work of Feireisl, Wen and Zhu~\cite{FeiWenZhu}, where the conditional regularity result for (deterministic)  Navier--Stokes--Fourier system with  bounded density and temperature has been indeed proved rigorously.

We mention that the concept of statistical strong solutions has been used by Lanthaler, Mishra and Par\'es-Pulido~\cite{LanMiPa,Pa} in the context of incompressible Euler system. For elliptic problems the corresponding literature is quite extensive, see, e.g., Barth~et al.~\cite{elliptic2}, Charrier et al.~\cite{elliptic1} and the references therein. In \cite{klingel, weber, Mishra_Schwab} convergence analysis of the Monte Carlo methods has been studied for scalar hyperbolic equations building on deterministic pathwise arguments.  In contrary to these works, we do not assume a priori the existence of statistical solution and the convergence of approximate solutions in random space will be proved using genuine stochastic compactness arguments.

This paper is organised as follows.
In Section~\ref{STA} we present statistical analysis of the Monte Carlo estimators for the expectation and deviation of a statistical strong solution to the Navier-Stokes-Fourier system.
Section~\ref{CON} is devoted to the convergence  of a finite volume method with random data under the assumption that the numerical density and temperature are bounded in probability.
Combining the results for the Monte Carlo sampling with those for the deterministic approximations, we derive the main results of this paper: convergence and error estimates of the Monte Carlo FV method for the random Navier--Stokes--Fourier system, see  Section~\ref{MOCA} and Section~\ref{sec_error}, respectively. In Section~\ref{num} we present numerical experiments to illustrate our theoretical results. Section~\ref{sec_con} closes the paper with concluding remarks.

%

\section{Statistical analysis of the Navier--Stokes--Fourier system}\label{STA}
To begin with statistical analysis  we introduce the space of the random data for \eqref{i1}-\eqref{i5}
\begin{align}
D = \Big\{ [\vr, \vu, \vt, \mu, \lambda, \kappa, \vc{g} ]\, \Big|	\, &  \vr \in W^{3,2}(\Td), \ \vu \in W^{3,2}(\Td;\R^d), \ \vt \in W^{3,2}(\Td), \ \min_{\Td}  \vr > 0,\
		\min_{\Td} \vt  > 0,\br
& \intTd{\left[ E - S - (c_v+1)\vr + 1 \right]} <  \infty, \ \mu > 0,\ \lambda \geq 0,\br
& \kappa  > 0, \ \vc{g} \in W^{2,2}(\Td; \R^d)\Big \},
\label{M2}
\end{align}
which is considered as a Borel subset of the Polish space
\[
X = W^{3,2}(\Td) \times W^{3,2}(\Td; \R^d)\times W^{3,2}(\Td) \times \R \times \R  \times \R  \times W^{2,2}(\Td; \R^d).
\]

Let us assume that the model data are random variables in $D$. More precisely, there is a complete probability space $[\Omega, \mathcal{B}, \mathcal{P}]$ and a measurable mapping
\begin{align*}
\vc{U}_0: \omega \in \Omega \longrightarrow  \Big[ \vr_0(\omega), \vu_0(\omega), \vt_0(\omega), \mu(\omega), \lambda(\omega), \kappa(\omega), \vc{g}(\omega) \Big] \in D
\end{align*}
for a.a. $\omega \in \Omega$.

Clearly, random data  lead to random Navier--Stokes--Fourier system \eqref{i1}--\eqref{i5}.
In what follows we want to analyse this random system
on a time interval $(0,T),$  where $T$ is a deterministic number. 
Note that in general a statistical strong solution $(\vr, \vu, \vt)$  exists on $[0, T_{max})$, where $T_{max}$ is a  maximal time of existence of a strong solution. As proved in \cite{feiluk_1}, $T_{max}$ is a random variable.
In Section~\ref{rnd_data} we will show that the statistical strong solution indeed exists on $[0,T],$ i.e.
$T < T_{max}(\omega)$ for a.a. $\omega \in \Omega.$   This will follow from our principal hypothesis \eqref{bip} on boundedness of density and temperature in probability. Due to the uniqueness of a strong solution, we denote by $(\vr, \vu, \vt)[\vc{U}_0, t]$ the strong solution of the Navier--Stokes--Fourier system corresponding to the data $\vc{U}_0 \in D$.



\subsection{Statistical convergence}

In this section we define Monte Carlo estimators of the expectation and deviation and
discuss their statistical convergence for the Navier--Stokes--Fourier system \eqref{i1}--\eqref{i5} in the regularity class given by  a priori estimates \eqref{space-1}.
Based on the deterministic estimates \eqref{space-1} we assume for data that
\begin{subequations}\label{Ass-B}
\begin{align}
& 
\expe{ \intTd{\left[ E(\vr_0, \vu_0, \vt_0) - S(\vr_0, \vt_0) - (c_v+1)\vr_0 + 1 \right]}}  < +\infty, \\
& \| \vc{g} \|_{C(\Td; \R^d)} \leq \Ov{g} \ \ \prst-\mbox{a.s.}, \hspace{1cm} \Ov{g} - \mbox{a deterministic constant}
\end{align}
\end{subequations}
and obtain the following convergence result. Here the notation $\expe{\cdot}$ means the expectation.

\begin{Proposition}[Strong law of large numbers] \label{PS1}
	Let $\vc{U}_0^n$, $n=1,2,\dots$ be independent, identically distributed (i.i.d.)
 copies of  random data
	\[
	\vc{U}_0 = \Big[ \vr_0, \vu_0, \vt_0, \mu , \lambda, \kappa, \vc{g} \Big] \in D \quad \mbox{ satisfying } \quad \eqref{Ass-B}.
	\]
Then it holds for all $t \in [0,T]$ that
\begin{multline}\label{eq-slln}
	\frac{1}{N} \sum_{n=1}^N X^n (t, \cdot) \to 0 \ \mbox{in}\ L^1(\Td; \R^{d+2})
 \quad \mbox{and}\quad
	\frac{1}{N} \sum_{n=1}^N Y^n (t, \cdot) \to 0 \ \mbox{in}\ L^1(\Td; \R^{d+2})
 \\ \mbox{as} \ N \to \infty, \quad \prst-\mbox{a.s.},
\end{multline}
where
\begin{equation}\label{XY}
\begin{aligned}
  &     X^n(t, \cdot)  :=(\vr^n, \vm^n, S^n)(t, \cdot)  - \expe{(\vr, \vm, S)(t, \cdot) }, \quad \vm^n = \vr^n\vu^n, \quad S^n=\vr^n  \log \left( \frac{(\vt^n)^{c_v}}{\vr^n}\right)
  \\
  &     Y^n(t, \cdot)  :=\abs{X^n} - \Dev{(\vr, \vm, S)(t, \cdot)},  \quad    \Dev{(\vr, \vm, S)} = \expe{\, \abs{(\vr, \vm, S)-\expe{(\vr, \vm, S)} } \, },
       \end{aligned}
\end{equation}
and $(\vr^n, \vu^n,\vt^n):=(\vr, \vu, \vt)[\vc{U}_0^n, t]$ is the strong solution to the Navier--Stokes--Fourier system \eqref{i1}--\eqref{i5} emanating from $\vc{U}_0^n$.	
\end{Proposition}

\begin{proof}
The proof follows from the Strong law of large numbers for random variables ranging in a separable Banach space, see Lemma \ref{LT710} (Ledoux and Talagrand \cite[Corollary 7.10]{LedTal}). Obviously, $(\vr^n, \vm^n,S^n)$ is a sequence of independent random variables. Consequently, $X^n$ and $Y^n$ are independent random variables with zero mean, i.e.
\begin{align*}
    \expe{X^n(t, \cdot)}=0 \quad \mbox{and}\quad   \expe{Y^n(t, \cdot)}=0.
\end{align*}
Moreover, using \eqref{Ass-B} we have
\begin{align*}
&\expe{\norm{(\vr,\vm,S)}_{L^{\infty}(0,T;L^1(\Td;\R^{d+2}))}}
\br
&\leq C( \Ov{g},T) \left( 1 + \expe{ \intTd{\left[ E(\vr_0, \vu_0, \vt_0) - S(\vr_0, \vt_0) - (c_v+1)\vr_0 + 1 \right]}}\right) < \infty,
\end{align*}
which gives
\begin{align*}
    \expe{\norm{X^n(t, \cdot)}_{L^1(\Td;\R^{d+2})}}<\infty \quad \mbox{and}\quad
        \expe{\norm{Y^n(t, \cdot)}_{L^1(\Td;\R^{d+2})}}<\infty.
\end{align*}
Application of  Lemma \ref{LT710} finishes the proof.
\end{proof}
\medskip

Further, we study the estimator of deviation
\begin{equation*}
 \frac1N \sum_{n=1}^N
\Big| (\vr^n,\vm^n,S^n)(t, \cdot)  - \frac1N \sum_{m=1}^N (\vr^m,\vm^m,S^m)(t, \cdot)\Big|,
\end{equation*}
which is used in the numerical simulations.
Specifically, applying the triangular inequality we obtain the following convergence of the deviation estimator
\begin{align}
&\left\| \frac1N \sum_{n=1}^N
\Big| (\vr^n,\vm^n,S^n)(t, \cdot)  - \frac1N \sum_{m=1}^N (\vr^m,\vm^m,S^m)(t, \cdot) \Big| \ - \Dev{(\vr,\vm,S)(t, \cdot)}
\right\|_{L^{1}(\Td;\R^{d+2})}
 \br
& \aleq \left\| \frac1N \sum_{n=1}^N
\Big| (\vr^n,\vm^n,S^n)(t, \cdot)  - \expe{(\vr,\vm,S)(t, \cdot)} \Big| \ - \Dev{(\vr,\vm,S)(t, \cdot)}
\right\|_{L^{1}(\Td;\R^{d+2})}
\br
& \quad + \left\|  \frac1N \sum_{n=1}^N \Big|
\frac1N \sum_{m=1}^N (\vr^m,\vm^m,S^m)(t, \cdot)  - \expe{(\vr,\vm,S)(t, \cdot)}   \Big|
\right\|_{L^{1}(\Td;\R^{d+2})}
\to 0
\label{C2}
\end{align}
as $N \to \infty, \prst-$a.s.

\subsection{Statistical covergence rate}
In this section we study statistical convergence rate of the Monte Carlo estimators. The key argument is the central limit theorem, cf. Lemma \ref{LT105} (\cite[Theorem~10.5]{LedTal}).

Applying the embedding theorem into a Hilbert space
$$
L^1(\Td) \hookrightarrow \hookrightarrow W^{-k,2}(\Td), \qquad k > d/2
$$
we can control the second moments with the expected value of the initial relative energy
\begin{align*}
& \expe{ \Big \| X^n (t, \cdot) \Big \|^2_{W^{-k,2}(\Td)} } +
\expe{ \Big \| Y^n \Big \|^2_{W^{-k,2}(\Td)} }
\\
& \aleq  \expe{ \Big \| (\vr,\vm,S) (t, \cdot) \Big \|^2_{W^{-k,2}(\Td)} } \aleq  \expe{ \Big \| (\vr,\vm,S) (t, \cdot) \Big \|^2_{L^1(\Td)} }
\\
& \leq  C( \Ov{g},T) \left( 1 +
\expe{ \left( \intTd{\left[ E(\vr_0, \vu_0, \vt_0) - S(\vr_0, \vt_0) - (c_v+1)\vr_0 + 1 \right]} \right)^2} \right).
\end{align*}
In order to control the right hand side of the above estimate, we need the following assumption
\begin{subequations}\label{Ass-C}
\begin{align}
&\expe{ \left( \intTd{\left[ E(\vr_0, \vu_0, \vt_0) - S(\vr_0, \vt_0) - (c_v+1)\vr_0 + 1 \right]} \right)^2} < \infty, \\
& \| \vc{g} \|_{C(\Td; \R^d)} \leq \Ov{g} \ \ \prst-\mbox{a.s.}, \hspace{1cm} \Ov{g} - \mbox{a deterministic constant.}
\end{align}
\end{subequations}
We point out that \eqref{Ass-C} implies \eqref{Ass-B}.
Now, we apply Lemma \ref{LT105} and obtain the statistical error estimates.
\begin{Proposition}[Central limit theorem]\label{PS2}
	Let $\vc{U}_0^n$, $n=1,2,\dots$ be i.i.d.  copies of  random data
	\[
	\vc{U}_0 = \Big[ \vr_0, \vu_0, \vt_0, \mu , \lambda, \kappa, \vc{g} \Big] \in D \quad \mbox{ satisfying } \quad \eqref{Ass-C}.
	\]
Then 
there hold 
\begin{multline*}
\frac{1}{\sqrt{N}} \sum_{n=1}^N  X^n(t, \cdot)  \to \big(\mathfrak{R},  \mathfrak{M}, \mathfrak{S}\big)
\quad \mbox{ and } \quad
\frac{1}{\sqrt{N}} \sum_{n=1}^N Y^n(t, \cdot) \to \big(\mathfrak{R}_D,  \mathfrak{M}_D, \mathfrak{S}_D\big)
\\
\mbox{ in law in }\ W^{-k,2}(\Td; \R^{d+2}), \ \mbox{as}\ N \to \infty, \quad
\mbox{ for all } t \in [0,T],
\end{multline*}	
where $X^n, Y^n$ are defined in \eqref{XY}, $\mathfrak{R},  \mathfrak{M}, \mathfrak{S}$, $\mathfrak{R}_D,  \mathfrak{M}_D, \mathfrak{S}_D$ are random Gaussian variables.
In particular, we have the convergence rate
\begin{equation}\label{C3a}
    N^{1/2} \left\| \frac1N \sum_{n=1}^N X^n (t, \cdot) \right\|_{W^{-k,2}(\Td;\R^{d+2})} \aleq 1
\quad \mbox{and} \quad
N^{1/2} \left\| \frac1N \sum_{n=1}^N Y^n (t, \cdot) \right\|_{W^{-k,2}(\Td;\R^{d+2})} \aleq 1
\end{equation}
for $k > \frac d 2,\, N= 1,2, \dots, \  \prst$-a.s.
\end{Proposition}

\bigskip

As shown in Proposition~\ref{PS1} the convergence of the Monte Carlo estimators holds in $L^1(\Td)$ topology. However, in order to recover a typical $N^{-1/2}$ statistical convergence rate of the Monte-Carlo method, we need to work in a weak topology $W^{-k,2}(\Td)$ due to low regularity estimates \eqref{space-1}.  Further, we obtain the convergence rate for the  deviation estimator
\begin{align}
&N^{1/2} \left\| \frac1N \sum_{n=1}^N
\Big| (\vr^n,\vm^n,S^n)(t, \cdot)  - \frac1N \sum_{m=1}^N (\vr^m,\vm^m,S^m)(t, \cdot) \Big| \ - \Dev{(\vr,\vm,S)(t, \cdot)}
\right\|_{W^{-k,2}(\Td;\R^{d+2})}
 \br
& \aleq  1
\quad \quad \mbox{ for } k > \frac d 2,\,  N = 1,2, \dots, \  \prst-\mbox{a.s}.
\label{C4a}
\end{align}

\section{Convergence and error estimates of a FV method}\label{CON}
Our next aim is to approximate the Navier--Stokes--Fourier system \eqref{i1}--\eqref{i5} in space and time. To this end we apply
the viscous finite volume method proposed by Feireisl et al.~\cite[Definition 2.3]{FLMS_FVNSF}, see Section~\ref{FV} for its
complete presentation.
We point out that the techniques and analysis below will not be limited to the specific numerical method, and can be extended to a broader class of consistent and stable approximation methods.

\subsection{Deterministic data} \label{det}
We start by presenting the convergence results of the FV method \eqref{scheme}. We note in passing that $\Delta t \searrow 0$ and $h \searrow 0$ are small positive parameters for time and space discretization, respectively. Recalling \cite[Theorem~5.6]{FLMS_FVNSF}, we have the following convergence for the deterministic data.
\begin{mdframed}[style=MyFrame]

\vspace{-0.5cm}
\begin{Proposition}[\textbf{Convergence}] \label{PF1}
Suppose that the initial data are regular belonging to the following spaces
\begin{equation*} 
	(\vr_0,\vu_0,\vt_0) \in W^{3,2}(\Td;\R^{d+2}),\ 0 < \underline{\vr} \leq \min_{\Td} \vr_0,\ 0 < \underline{\vt} \leq \min_{\Td} \vt_0.
\end{equation*}
Let  $\vc{g} \in W^{2,2}(\Td; \R^d)$, $\mu > 0$, $\lambda \geq 0$, and $\kappa >0$.
Let  $\left(\vrh, \vuh, \vth\right)_{h \searrow 0}$ be a family of numerical solutions obtained by the FV method \eqref{scheme} with $\TS \approx h$ satisfying
	\begin{equation} \label{F4bis}
	 \sup_{h} \left(\| \vrh,\vrh^{-1},\vth,\vth^{-1} \|_{L^\infty((0,T)\times \Td;\R^4)} \right)< \infty.
	\end{equation}
Then
\begin{align*}
&  \left\| \big( \vrh, S_h \big) - \big( \vr, S \big)\right\|_{L^p( (0,T) \times \Td; \R^2)} + \left\| \vm_h - \vm \right\|_{L^p(0,T; L^2(\Td;\R^d))}
\quad \to 0 \ \ \  \mbox{as}\ h \to 0
\end{align*}
for any $1 \leq p < \infty$, where	$(\vr, \vu,\vt)$ is a
global classical solution  to the Navier--Stokes--Fourier system \eqref{i1}--\eqref{i5}; $\vm = \vr \, \vu$ and $S= \vr \log \left( \frac{\vt^{c_v}}{\vr}\right)$ are the corresponding momentum and entropy.
\end{Proposition}

\vspace{-0.2cm}
\end{mdframed}

This result is based on the weak-strong uniqueness principle for the Navier--Stokes--Fourier system \eqref{i1}--\eqref{i5}.
Indeed, in general the FV method only converges weakly* to a dissipative weak solution, see \cite{FLMS_FVNSF}. Note that this convergence is conditional and requires boundedness of numerical densities and temperatures \eqref{F4bis}. Further,  due to the dissipative weak-strong uniqueness principle
a strong solution is stable in the class of dissipative solutions if \eqref{F4bis} holds, see~\cite{brezina}. Thus, as long as a strong solution exists, FV solutions converge strongly to the strong solution. Now, applying
the conditional regularity result due to
Feireisl, Wen, and Zhu \cite{FeiWenZhu} the strong solution is global in time, since density and temperature are bounded on $[0,T]$. The $W^{3,2}$ regularity of initial data is inherited by the strong solution, too.

Further, assuming slightly higher regularity of data we can derive convergence rate of the FV method by means of the relative energy method, see our recent work \cite[Theorem 5.2]{BLMSY}.


\begin{mdframed}[style=MyFrame]

\vspace{-0.5cm}
\begin{Proposition}[\textbf{Error estimates}] \label{PF2}
 Let the initial data $(\vr_0, \vu_0, \vt_0)$ belong to the regularity class
\begin{equation*}
	(\vr_0,\vu_0,\vt_0) \in W^{6,2}(\Td;\R^{d+2}),\   \min_{\Td} \vr_0 = \underline{\vr} > 0,\
		\min_{\Td} \vt_0 = \underline{\vt} > 0
\end{equation*}
and  $\vc{g} \in W^{5,2}(\Td; \R^d)$, $\mu > 0$, $\lambda \geq 0$, and $\kappa > 0$.
Let $(\vrh, \vuh,\vth)$ be a numerical solution resulting from the FV method \eqref{scheme} with $(\TS, h) \in (0,1)^2$ satisfying \eqref{F4bis}.

Then the following estimates hold
\begin{equation}\label{CR2}
\norm{ \big( \vrh, \vm_h, S_h \big) - \big(\vr, \vm, S \big)}_{L^\infty(0,T; L^2(\Td;\R^{d+2}))}
\leq C( \TS^{1/2}+ h^{1/4}),
\end{equation}
where \[C=C( T, \| \vc{g} \|_{W^{5,2}(\Td; \R^d)},  \| (\vr_0, \vu_0,\vt_0)\|_{W^{6, 2}(\Td; \R^{d+2})},
\| (\vrh, \vrh^{-1} ,\vth, \vth^{-1}) \|_{L^\infty((0,T)\times \Td; \R^4)}, \Un{\vr}, \Un{\vt}).\]
	\end{Proposition}
	
\vspace{-0.7cm}
\end{mdframed}


\subsection{Random data}
\label{rnd_data}
We are now ready to consider random initial data $(\vr_0, \vu_0, \vt_0) $ and random parameters $(\mu , \lambda, \kappa, \vc{g})$ belonging to the set $D$.
 We start with discussing the measurability of the FV approximations. The FV method \eqref{scheme} is a time-implicit method. Consequently, one needs to solve a nonlinear system and might possibly get non-unique approximate solutions.
Applying Bensoussan and Temam \cite[Theorem A.1]{BenTem}  and \cite[Section 3.2]{FLSY_MC1}
there is a  {\em measurable selection}, specifically a Borel mapping, such that
\[
\Big[ \vr_0, \vu_0, \vt_0 , \mu, \lambda, \kappa, \vc{g} \Big] \in D \mapsto
(\vrh, \vuh, \vth) \in \Vrmh.
\]
Here $\Vrmh$ represents the set of all possible FV approximations those correspond to the data $\Big[ \vr_0, \vu_0, \vt_0 , \mu, \lambda, \kappa, \vc{g} \Big]$  for a fixed mesh discretisation parameter $h$.
Consequently, FV solutions $(\vrh, \vuh, \vth)$ under above selection are Borel measurable functions of the data.

As already discussed in Section~\ref{det} boundedness of numerical densities and temperatures is crucial in order to obtain the convergence to a global-in-time strong solution of the Navier--Stokes--Fourier system \eqref{i1}--\eqref{i5}. In statistical analysis we will weaken this assumption and only assume their boundedness in probability.
\medskip

\begin{mdframed}[style=MyFrame]
\textbf{Assumption}\ [{\bf Boundedness in probability of FV approximations}]

\noindent We assume that a  sequence of FV solutions {$\left(\vrh, \vuh, \vth \right)_{h\searrow 0}$} is \emph{bounded in probability}, i.e.
		\begin{align}
		&\mbox{for any}\ \ep > 0, \ \mbox{there exists}\ M= M(\ep) \ \mbox{such that}
		\mbox{ for all } h \in (0,1)
        \br & {\prst\left(\left[  \| \vrh, \vrh^{-1}, \vth, \vth^{-1}\|_{L^\infty((0,T)\times \Td;\R^4) }  > M
		\right]\right) \leq \ep.} \label{bip}
	\end{align}
\end{mdframed}
\medskip

Our next goal is to prove the convergence of random numerical solutions $\left(\vrh(\omega), \vuh(\omega), \vth(\omega) \right)_{h\searrow 0},$ $\omega \in \Omega$ obtained by the FV method to a strong statistical solution $(\vr, \vu, \vt)$.  Applying intrinsic stochastic compactness arguments, based on the Skorokhod representation theorem \cite{Jakub} and Gy\" ongy--Krylov theorem \cite{Gkrylov} we obtain the following convergence result, see also
\cite{FeiLuk2021}.

\begin{enumerate}
	
	\item We consider numerical solutions   $(\vrh, \vuh, \vth)$
Borel measurable with respect to the data
	\[
	\vc{U}_0=\Big[ \vr_0, \vu_0, \vt_0, \mu, \lambda, \kappa, \vc{g} \Big] \in D \quad \mbox{ satisfying } \quad \eqref{Ass-B}.
	\]

	\item
	
	Taking a subsequence of FV solutions $(\vrhk, \vu_{h_k}, \vt_{h_k})_{h_k \searrow 0}$ we consider a family of random variables
	\begin{align*}
	& \Big[ \vr_0, \vu_0, \vt_0, \mu, \lambda,\kappa, \vc{g}, \vrhk,  \vm_{h_k}:= \vrhk \vu_{h_k}, S_{h_k}:=\vrhk  \log \left( \frac{(\vt_{h_k})^{c_v}}{\vrhk}\right), \Lambda_{h_k} \Big]_{h_k \searrow 0}
	\br
	& \mbox{with } \  \Lambda_{h_k} = \| (\vrhk, \vrhk^{-1}, \vthk, \vthk^{-1}) \|_{L^\infty ((0,T) \times \Td; \R^{4})}
	\end{align*}
	ranging in the Polish space
	\begin{eqnarray*}
	&& Y = W^{3,2} (\Td) \times W^{3,2} (\Td; \R^d) \times  W^{3,2} (\Td) \times \R \times \R \times \R \times W^{2,2}(\Td; \R^d)
	\times W^{-k,2}((0,T) \times \Td)  \br
      && \hspace{1.5cm} \times W^{-k,2}((0,T) \times \Td; \R^d)  \times W^{-k,2}((0,T) \times \Td) \times \R, \quad   k > d/2.
	\end{eqnarray*}
	
	In view of {hypothesis \eqref{bip}}, the family of laws associated to
	\[
	\Big[ \vr_0, \vu_0, \vt_0, \mu, \lambda, \kappa, \vc{g}, \vrhk,   \vm_{h_k}, S_{h_k},  \Lambda_{h_k} \Big]_{h_k \searrow 0}
	\]
is tight in $Y$. Applying the Skorokhod representation theorem \cite{Jakub} we conclude that there is a new probability space and a new sequence of random variables
\[
\Big[ \tvr_{0,h_k},  \tvu_{0,h_k}, \tvt_{0,h_k}, \widetilde\mu_{h_k}, \widetilde \lambda_{h_k},\widetilde \kappa_{h_k}, \widetilde{ \vc{g}}_{h_k}, \widetilde{\vr}_{h_k},   \widetilde{\vm}_{h_k}, \widetilde{S}_{h_k}, \widetilde{\Lambda}_{h_k} \Big]	
\sim \Big[ \vr_0, \vu_0, \vt_0, \mu, \lambda, \kappa, \vc{g}, \vrhk,   \vm_{h_k}, S_{h_k},  \Lambda_{h_k} \Big]
\]
satisfying
\begin{align*}
\widetilde{\Lambda}_{h_k} &= {\| (\widetilde{\vr}_{h_k} , \widetilde{\vr}_{h_k}^{-1} , \widetilde{\vt}_{h_k} , \widetilde{\vt}_{h_k}^{-1}) \|_{L^\infty((0,T)\times \Td; \R^{4})} < \infty,} \  \br
\tvr_{0,{h_k}} &\to \tvr_0, \ \tvt_{0,{h_k}} \to \widetilde{\vt}_0 \ \mbox{in}\ W^{3,2}(\Td), \br
\tvu_{0,{h_k}} &\to \widetilde{\vu}_0 \ \mbox{in}\ W^{3,2}(\Td; \R^d), \br
\widetilde{\mu}_{h_k} &\to \widetilde{\mu},\ \widetilde{\lambda}_h \to \widetilde{\lambda},
\ \widetilde{\kappa}_h \to \widetilde{\kappa}, \
\widetilde{\vc{g}}_{h_k} \to \widetilde{\vc{g}} \ \mbox{in}\ C(\Td; \R^d) , \br
\tvr_{h_k} &\to \tvr, \  \widetilde{S}_{h_k} \to  \widetilde{S} \ \mbox{in}\ L^r((0,T) \times \Td),\  1 \leq r < \infty,  \br
 \widetilde{\vm}_{h_k} & \to  \widetilde{\vm} \ \mbox{in}\ L^r(0,T; L^2 (\Td; \R^d)),\ 1 \leq r < \infty, 
\end{align*}
a.s.,
where thanks to Proposition~\ref{PF1} $ (\tvr, \tvu, \tvt)$ is the strong solution of the Navier--Stokes--Fourier system \eqref{i1}--\eqref{i5} corresponding to the data
\[
\Big[ \tvr_0, \tvu_0, \tvt_0, \widetilde{\mu}, \widetilde{\lambda},  \widetilde{\kappa},\widetilde{\vc{g}} \Big]
\sim \Big[ \vr_0, \vu_0, \vt_0, \mu, \lambda,\kappa, \vc{g} \Big] .
\]
The symbol $\sim$ represents equivalence in the law of random variables.

\item

 Due to the uniqueness of a strong solution, there is no need to extract a subsequence as the limit is unique.
More importantly, by means of the Gy\" ongy--Krylov theorem \cite{Gkrylov} we recover the convergence in the original
probability space,
\begin{align}\label{F6}
	\| (\vrh, S_h)- (\vr, S) \|_{L^p((0,T) \times \Td; \R^2)} &\to 0  \ \mbox{in probability}, \br
	\| \vm_h - \vm \|_{L^p(0,T; L^2( \Td; \R^d))} &\to 0\  \mbox{in probability} 
\end{align}
for any $1 \leq p < \infty$,  where $(\vr, \vu,\vt):=(\vr, \vu, \vt)[\vc{U}_0, t]$ is the strong solution to the Navier--Stokes--Fourier system \eqref{i1}--\eqref{i5} emanating from $\vc{U}_0$; $\vm = \vr \vu$ and $S= \vr \log \left( \frac{\vt^{c_v}}{\vr}\right)$ are the corresponding momentum and entropy.
	\end{enumerate}

Having obtained the convergence of random numerical solution $ (\vrh(\omega), \vuh(\omega), \vth(\omega) ),\ \omega \in \Omega,$ we are ready to show the convergence of the statistical moments obtained by the Monte Carlo FV method.

\section{Convergence of the Monte Carlo FV method}
\label{MOCA}
Let us start by splitting the error of the Monte Carlo FV approximation in the following way
\begin{align*}
&\frac{1}{N} \sum_{n = 1}^N  \big(\vr_h^n, \vm_h^n, S_h^n \big)  - \expe{\big(\vr,\vm,S \big)} \br
=& \frac{1}{N} \sum_{n = 1}^N  \big(\vr^n, \vm^n, S^n \big)  - \expe{\big(\vr,\vm,S \big)} + \frac{1}{N} \sum_{n = 1}^N  \bigg[ \big(\vr_h^n, \vm_h^n, S_h^n \big)  -\big(\vr^n, \vm^n, S^n \big)  \bigg].
\end{align*}
Combining the statistical estimates \eqref{eq-slln}, \eqref{C2} and the deterministic convergence analysis \eqref{F6}, we obtain the following convergence results.

\begin{mdframed}[style=MyFrame]
	\begin{Theorem}[\textbf{Convergence}] \label{FVT1}
		Let the data  be random variables
\[
	\vc{U}_0=\Big[ \vr_0, \vu_0, \vt_0, \mu, \lambda, \kappa, \vc{g} \Big] \in D \quad \mbox{ satisfying } \quad \eqref{Ass-B}.
	\]	
Suppose that $\vc{U}_0^n = [\vr_0^n, \vu_0^n, \, \vt_0^n, \, \mu^n, \lambda^n,  \, \kappa^n, \, \vc{g}^n]$, $n=1,2,\dots, N$ are i.i.d. copies of  random data.
Let $\left(\vrh^n, \vuh^n, \vth^n \right)_{h \searrow 0}$
be a sequence of FV solutions \eqref{scheme} corresponding to these data samples.
Assume that FV solutions $\left(\vrh^n, \vuh^n, \vth^n\right)_{h\searrow 0},$  $n=1,2,\dots$
are bounded in probability, cf. \eqref{bip}.

Then for the Monte Carlo FV estimators of the expectation and deviation
\begin{align*}
\frac1N\sum_{n = 1}^N (\vrh^n, \vm_h^n, S_h^n),
\quad  \quad
 \frac1N \sum_{n=1}^N
\Big| (\vrh^n,\vm_h^n,S_h^n)(t, \cdot)  - \frac1N \sum_{m=1}^N (\vrh^m,\vm_h^m,S_h^m)(t, \cdot)\Big|
\end{align*}
we have that for any $\ep > 0$ there hold

\begin{align*} 
 &\lim_{h \to 0}\prst \Bigg[  \lim_{N \to \infty}  \left\| \frac{1}{N} \sum_{n = 1}^N (\vrh^n, \vm_h^n, S_h^n) - \expe{\big(\vr,\vm,S \big)} \right\|_{L^p(0,T; L^1(\Td;\R^{d+2}))}  \leq \ep \Bigg]  = 1,
\\
& \lim_{h \to 0}\prst \Bigg[   \lim_{N \to \infty} \left\|  \frac1N \sum_{n=1}^N
\Big| (\vrh^n, \vm_h^n, S_h^n)  - \frac1N \sum_{m=1}^N (\vrh^m, \vm_h^m, S_h^m) \Big| - \Dev{\big(\vr,\vm,S \big)} \right\|_{L^p(0,T; L^1(\Td;\R^{d+2}))}
\br
&\hspace{2cm} \leq \ep \Bigg]  = 1 
\end{align*}
for $p \in [1,\infty)$.
\end{Theorem}
\end{mdframed}

\begin{proof}
In what follows we only present the proof for the density $\vr$ as the same procedure can be done for the momentum $\vm$ and the entropy $S$.
\begin{itemize}
\item  Monte Carlo FV expectation estimator: According to the convergence in probability \eqref{F6} we have that for any $\ep > 0$ it holds
\begin{align*}
 \lim_{h \to 0}\prst \Bigg[   \left\| \frac{1}{N} \sum_{n = 1}^N \big(\vrh^n - \vr^n \big)\right\|_{L^p((0,T) \times \Td)}  \leq \ep \Bigg]  = 1 \quad \mbox{ for every}\ N=1,2,\dots,\  p \in [1,\infty).
\end{align*}
Applying the triangle inequality
\begin{align*}
& \prst \Bigg[   \left\| \frac{1}{N} \sum_{n = 1}^N \vrh^n - \expe{\vr} \right\|_{L^p((0,T) \times \Td)}  \leq \ep \Bigg]
\br
& \geq  \prst \Bigg[  \left\| \frac{1}{N} \sum_{n = 1}^N \big(\vrh^n - \vr^n \big)\right\|_{L^p((0,T) \times \Td)} +  \left\| \frac{1}{N} \sum_{n = 1}^N \vr^n - \expe{\vr} \right\|_{L^p((0,T) \times \Td)}  \leq \ep \Bigg]
\end{align*}
and $\prst$-a.s. convergence result \eqref{eq-slln}
\begin{align*}
 \lim_{N \to \infty} \left\| \frac{1}{N} \sum_{n = 1}^N \vr^n - \expe{\vr} \right\|_{L^p(0,T; L^1(\Td))} \to 0  \hspace{1cm}
\mbox{ for } \ p \in [1,\infty),\ \quad \prst-\mbox{a.s.}
\end{align*}
we obtain that for any $\ep > 0$ it holds
\begin{align*}
 \lim_{h \to 0}\prst \Bigg[   \lim_{N \to \infty} \left\| \frac{1}{N} \sum_{n = 1}^N \vrh^n - \expe{\vr} \right\|_{L^p(0,T; L^1(\Td))}  \leq \ep \Bigg]  = 1 \quad \mbox{ for }\  p \in [1,\infty).
\end{align*}

\item  Monte Carlo FV deviation estimator: Analogously to the expectation estimator,  using \eqref{F6} we obtain for any $\ep > 0$ 
\begin{align*}
& \lim_{h \to 0}\prst \Bigg[  \left\| \left(\vrh^n - \frac{1}{N} \sum_{m = 1}^N \vrh^m \right) - \left(\vr^n - \frac{1}{N} \sum_{m = 1}^N \vr^m \right) \right\|_{L^p((0,T) \times \Td)}  \leq \ep \Bigg]  = 1 \br
& \hspace{3cm}  \mbox{ for every} \ N=1,2,\dots,\ n \leq N, \ p \in [1,\infty),
\end{align*}
which gives
\begin{align*}
1 \geq & \lim_{h \to 0}\prst \Bigg[  \left\| \frac{1}{N} \sum_{n = 1}^N \left(\vrh^n - \frac{1}{N} \sum_{m = 1}^N \vrh^m \right) - \frac{1}{N} \sum_{n = 1}^N\left(\vr^n - \frac{1}{N} \sum_{n = 1}^N \vr^m \right) \right\|_{L^p((0,T) \times \Td)}  \leq \ep \Bigg]
\br
\geq &  \lim_{h \to 0}\prst \Bigg[   \frac{1}{N} \sum_{n = 1}^N  \left\| \left(\vrh^n - \frac{1}{N} \sum_{m = 1}^N \vrh^m \right) - \left(\vr^n - \frac{1}{N} \sum_{n = 1}^N \vr^m \right) \right\|_{L^p((0,T) \times \Td)}  \leq \ep \Bigg]  \geq 1.
\end{align*}
Together with the triangle inequality and the $\prst$-a.s. convergence result \eqref{C2} we have
\begin{align*}
 \lim_{h \to 0}\prst \Bigg[   \lim_{N \to \infty} \left\|  \frac1N \sum_{n=1}^N
\Big| \vrh^n  - \frac1N \sum_{m=1}^N \vrh^m\Big| - \Dev{\vr} \right\|_{L^p(0,T; L^1(\Td))}  \leq \ep \Bigg]  = 1 \quad \mbox{ for }\  p \in [1,\infty).
\end{align*}
\end{itemize}

\end{proof}

\section{Error estimates of the Monte Carlo FV method}
\label{sec_error}
The aim of this section is to analyse the convergence rate of the Monte Carlo FV approximations.
To this end, let us consider more regular random data, i.e.
\begin{align*}
& \vr_0 \in W^{6,2}(\Td), \ \vu_0 \in W^{6,2}(\Td;\R^d), \ \vt_0 \in W^{6,2}(\Td), \ \min_{\Td}  \vr_0 > 0,\
		\min_{\Td} \vt_0  > 0, \  \br
&\mu > 0,\ \lambda \geq 0,\ \kappa  > 0, \ \vc{g} \in W^{5,2}(\Td; \R^d),\  \ \ \prst-\mbox{a.s.}
\end{align*}
Under the assumption \eqref{bip} that the FV solutions $(\vrh, \vuh, \vth)_{h \searrow 0}$ are bounded in probability, it follows from the arguments of Section \ref{rnd_data} that there exists
a random classical solution  $(\vr, \vu, \vt)$ of the Navier--Stokes--Fourier system \eqref{i1}--\eqref{i5}, such that
\begin{eqnarray}
&&(\vr,\vu,\vt) \in  C([0,T]; W^{6,2}(\Td;\R^{d+2})) \cap C^1([0,T] \times \Td;\R^{d+2}), \qquad  \prst-\mbox{a.s.} \label{reg_class}
\end{eqnarray}
and the numerical solutions converge to this strong solution in probability. Note that higher deterministic regularity is a classical result for local strong solutions \cite{KawShi, BreFeiHof}. We refer to our recent work \cite[Proposition~2.1]{BLMSY} for the global regularity result in the class \eqref{reg_class} obtained for bounded density and temperature.
Combining the statistical estimates \eqref{C3a}, \eqref{C4a} and the deterministic error estimates \eqref{CR2}, we obtain  a priori error estimates for the Monte Carlo FV approximations.

\begin{mdframed}[style=MyFrame]
\begin{Theorem} [{\bf Error estimates}] \label{FVT2}
Let the data  be random variables
\begin{align*}
&\vc{U}_0 = \Big[ \vr_0, \vu_0, \vt_0, \mu , \lambda, \kappa, \vc{g} \Big] \in D  \mbox{ satisfying } \quad \eqref{Ass-C}
\br
& \mbox{and } \ (\vr_0,\vu_0,\vt_0) \in W^{6,2}(\Td;\R^{d+2}), \ \vc{g}  \in W^{5,2}(\Td;\R^d), \ \prst-\mbox{a.s.}
\end{align*}
Suppose that $\vU_0^n = [\vr_0^n, \vu_0^n, \, \vt_0^n, \, \mu^n, \lambda^n,  \, \kappa^n, \, \vc{g}^n]$, $n=1,2,\dots, N$ are i.i.d. copies of  random data.
Let $\left(\vrh^n, \vuh^n, \vth^n \right)_{h \searrow 0}$
be a sequence of FV solutions \eqref{scheme} corresponding to these data samples.
Assume that FV solutions $\left(\vrh^n, \vuh^n, \vth^n\right)_{h\searrow 0},$  $n=1,2,\dots$
are bounded in probability, cf. \eqref{bip}.

Then the Monte Carlo FV expectation and deviation estimators
\begin{align*}
\frac1N\sum_{n = 1}^N (\vrh^n, \vm_h^n, S_h^n),
\quad  \quad
 \frac1N \sum_{n=1}^N
\Big| (\vrh^n,\vm_h^n,S_h^n)(t, \cdot)  - \frac1N \sum_{m=1}^N (\vrh^m,\vm_h^m,S_h^m)(t, \cdot)\Big|
\end{align*}
satisfy for every $N=1,2,\dots$ and all $t \in [0,T]$ the following error estimates: for any $\ep > 0$, there exists $K=K(\ep)$ such that
\begin{align}  \label{EE-1}
&\prst \Bigg[   \left\| \frac{1}{N} \sum_{n = 1}^N  \big(\vrh^n, \vm_h^n, S_h^n \big) (t, \cdot) - \expe{(\vr, \vm, S \big) (t, \cdot)}  \right\|_{W^{-k,2}(\Td;\R^{d+2})}
\br
& \hspace{1cm} \leq K \left( \Delta t^{1/2} + h^{1/4}+ N^{-1/2} \right)\Bigg] \geq 1 - \ep,
  \\
 & \prst \Bigg[  \left\|  \frac1N \sum_{n=1}^N
\Big| (\vrh^n, \vm_h^n, S_h^n)  - \frac1N \sum_{m=1}^N (\vrh^m, \vm_h^m, S_h^m) \Big| - \Dev{\big(\vr,\vm,S \big)} \right\|_{W^{-k,2}(\Td;\R^{d+2})}
\br
& \hspace{1cm}\leq K \left( \Delta t^{1/2} + h^{1/4}+ N^{-1/2} \right)\Bigg]  \geq 1 - \ep. \label{EE-2}
\end{align}
\end{Theorem}
\end{mdframed}

\medskip

\begin{Remark}
We point out that the convergence rate with respect to the mesh parameter $h$ is only $1/4$. Under the assumption that numerical density and temperature are bounded from below and above,  the discrete semi-norm of the velocity gradient is not bounded, which limits the convergence rate, see \cite{BLMSY} for more details.
\end{Remark}

\begin{proof}
Here we only prove \eqref{EE-1} as \eqref{EE-2} can be done in the same way.
Proposition~\ref{PF2} gives us the estimates on the approximation error: for any $\ep > 0$, there exists $K = K(\ep)$ such that
\begin{equation*}
\prst \Bigg[  \left\| \big(\vrh^n, \vm_h^n, S_h^n \big) (t, \cdot) - \big(\vr^n, \vm^n, S^n \big) (t, \cdot)  \right\|_{L^2(\Td)}
\leq K( \Delta t^{1/2} + h^{1/4}) \Bigg]  \geq 1 - \ep \ \mbox{ for all } \ t \in [0,T].
\end{equation*}
Indeed, $K$ only depends on $\ep$, that is to say, it is independent of $n$. Consequently, we have: for any $\ep > 0$, there exists $K=K(\ep)$ such that
\begin{align*}
&\prst \Bigg[   \left\| \frac{1}{N} \sum_{n = 1}^N \bigg( \big(\vrh^n, \vm_h^n, S_h^n \big) (t, \cdot) - \big(\vr^n, \vm^n, S^n \big) (t, \cdot) \bigg) \right\|_{L^{2}(\Td;\R^{d+2})}  \leq K( \Delta t^{1/2} + h^{1/4} ) \Bigg]  \geq 1 - \ep.
\end{align*}
On the other hand, from \eqref{C3a} we get
\begin{align*}
&\prst \Bigg[   \left\| \frac{1}{N} \sum_{n = 1}^N  \big(\vr^n, \vm^n, S^n \big) (t, \cdot) - \expe{(\vr, \vm, S \big) (t, \cdot)}  \right\|_{W^{-k,2}(\Td;\R^{d+2})}  \aleq  N^{-1/2} \Bigg]  =1.
\end{align*}
Combining above two formula we finish the proof.
\end{proof}

\section{Numerical experiment}\label{num}
In this section we illustrate the obtained theoretical convergence results of the Monte Carlo FV method by means of numerical simulations.
First, we define the following errors:
\begin{itemize}
\item Error of mean value:
\begin{align}\label{E1}
&E_1(U_h) =  \frac1{M} \sum_{m=1}^M \left(  \left\| \frac{1}{N} \sum_{n = 1}^N U_h^{n,m} (T, \cdot)- \expe{ U (T, \cdot) }  \right\|_{L^p(\Td)}  \right),
\\ \label{E2}
& E_2(U_h) = \left[ \frac1{M} \sum_{m=1}^M \left( \left\| \frac{1}{N} \sum_{n = 1}^N  U_h^{n,m}(T, \cdot) - \expe{U(T, \cdot)}  \right\|^{p}_{L^{p}(\Td)}\right) \right]^{1/p};
\end{align}

\item Error of deviation:
\begin{align}\label{E3}
& E_3(U_h ) =
\frac1{M} \sum\limits_{m=1}^M \left(   \left\| \frac 1 N \sum\limits_{n=1}^N \Big| U_h^{n,m}(T, \cdot)   - \frac{1}{N}  \sum\limits_{l=1}^N U_h^{l,m}(T, \cdot)   \Big| \ -  \Dev{U(T, \cdot) }  \right\|_{L^p(\Td)}  \right),
\\ \label{E4}
& E_4(U_h) =
\left[ \frac1{M} \sum\limits_{m=1}^M \left( \left\| \frac 1 N \sum\limits_{n=1}^N \Big| U_h^{n,m}(T, \cdot)   - \frac{1}{N}  \sum\limits_{l=1}^N U_h^{l,m}(T, \cdot)   \Big| \ -  \Dev{U(T, \cdot)}   \right\|_{L^p(\Td)}^p  \right) \right]^{1/p}.
\end{align}
\end{itemize}
Here, $U \in \{ \vr, \vm, S \}$, and $U_h^{n,m}$ is the FV approximation obtained with the $m$-th realisation of the $n$-th copy of the random data. Note that the exact expectation and deviation are approximated by the numerical solutions on the finest grid (with $h_{ref}$) emanating from $N_S$ copies of random data
\begin{align*}
&\expe{U(t,x)} =  \frac1{N_S} \sum_{n=1}^{N_S}  U^{n}_{h_{ref}}(t,x),
\quad
\Dev{U(t, x)}  =  \frac1{N_S} \sum_{n=1}^{N_S}  \Bigg|U_{h_{ref}}^{n}(t,x) - \frac1{N_S} \sum_{n=1}^{N_S}  U_{h_{ref}}^{n}(t,x) \Bigg|.
\end{align*}

Further, in order to validate our theoretical results,  we present the statistical errors and the total errors by selecting different $h$ and $N$ inside \eqref{E1}--\eqref{E4}:
\begin{itemize}
\item Statistical errors with respect to $N$ for a fixed mesh size $h=h_{ref}$:
\begin{align*}
&
N = 10 \cdot 2^n, \quad n = 0,\dots,3.
\end{align*}
\item Total errors with respect to the pair $(h,1/N) = (h,\mathcal{O}(h))$:
\begin{align*}
& h= 2/(32 \cdot 2^n), \quad N = 10 \cdot 2^n, \quad n = 0,\dots,3.
\end{align*}
\end{itemize}
Note that instead of theoretical convergence results in probability, cf.~Theorems~\ref{FVT1}, \ref{FVT2}, we consider in our numerical simulations the convergence in expectation that is more convenient for  practical purpose. Indeed, $\prst$-a.s. convergence presented in Propositions~\ref{PS1}, \ref{PS2}, and the boundedness of the expected values  imply convergence in the expectation for sampled (exact) solution.
In the simulations, the parameter $\eps$ in the FV method \eqref{scheme} is set to $0.6$. The parameters in  \eqref{E1}--\eqref{E4} are taken as $p = 1,2$, $M=40$, and $N_S=1000$, $h_{ref}:= 2/(32 \cdot 2^4)$.
\medskip

We consider the initial data to be a random perturbation of a vortex on $[-1,1]^2$
\begin{align*}
& \vr_0(x) = 1 + 0.1 \cdot Y_1(\omega) \cos(2\pi (x_1+x_2)), \\
& \vu_0(x) = 0.1 \cdot (Y_2(\omega) , Y_3(\omega) )^t +  \begin{cases}
\left(\frac{ [1-\cos(4\pi |x|)] x_2}{|x|},\ -\frac{ [1-\cos(4\pi |x|)]x_1}{|x|} \right)^t, & \mbox{if}\  |x| < 0.5,\\
(0,0)^t, & \mbox{otherwise},
\end{cases} \\
& \vt_0(x) = 2 + 0.2 \cdot Y_4(\omega) \sin(2\pi (x_1+x_2)).
\end{align*}
Other parameters in the Navier--Stokes--Fourier equations \eqref{i1}--\eqref{i5} are taken as
\begin{align*}
\vc{g} \equiv 0, \quad \tilde{\mu} = \mu + 10^{-4} \cdot Y_5(\omega), \quad
\tilde{\lambda} = \lambda + 10^{-4} \cdot Y_6(\omega), \quad
\tilde{\kappa} = \kappa + 10^{-4} \cdot Y_7(\omega)
\end{align*}
with $\mu =  \lambda = \kappa = 0.001$ and $Y_j(\omega) \ \overset{i.i.d.}{\sim}\ \mathcal{U}\left( -1, 1\right),\, j = 1,\dots, 7$.

Figure \ref{fig1} displays the mean and deviation of the numerical solutions $\vr, \vm, S$, as well as the zoom-in along the line $x=y$, obtained with the Monte Carlo FV method at $T = 0.1$. The statistical errors and total errors of $\vr,\vm,S, \vu, \vt$ in the $L^1$-norm, i.e.~$p = 1$ in \eqref{E1}-\eqref{E4}, are presented in Figure \ref{fig2} and \ref{fig3}, respectively.
Corresponding results in the $L^2$-norm are shown in Figure \ref{fig4} and \ref{fig5}, respectively.

The numerical results confirm that the statistical errors of the means of $\vr,\vm, S$, and their deviations converge with the rate $N^{-1/2}$. Further, the total errors obtained with the parameter pair $(h,1/N) = (h,\mathcal{O}(h))$ converge with a rate belonging to $[1/2, 1]$. 
This numerical experiment indicates a higher convergence rate than those obtained in theoretical part. Note however that theoretical convergence rates were obtained for general data.

\begin{figure}[htbp]
	\setlength{\abovecaptionskip}{0.cm}
	\setlength{\belowcaptionskip}{-0.cm}
	\centering
	\begin{subfigure}{0.32\textwidth}
		\includegraphics[width=\textwidth]{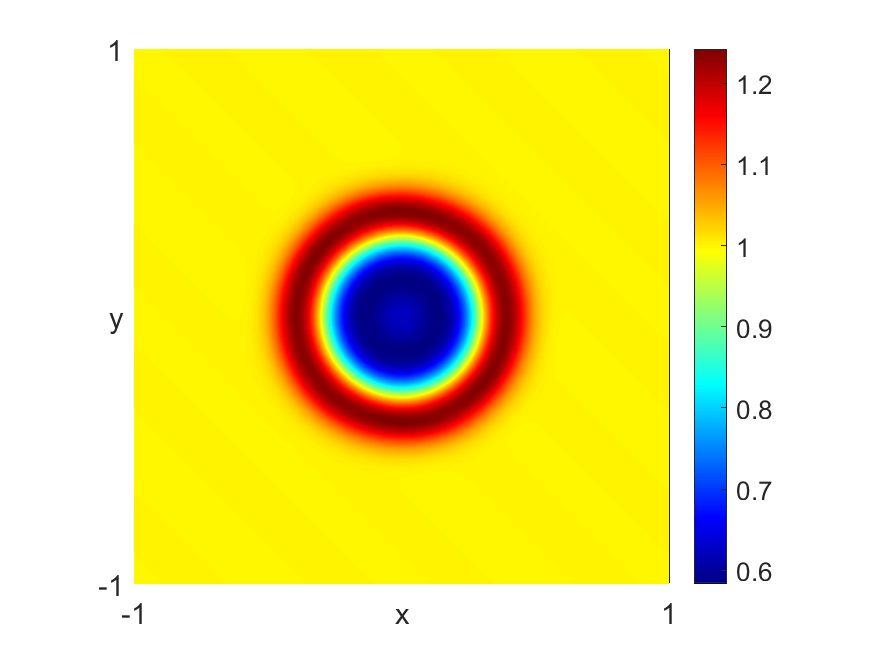}
		\caption{ \bf $\vr$ - Mean}
	\end{subfigure}	
	\begin{subfigure}{0.32\textwidth}
		\includegraphics[width=\textwidth]{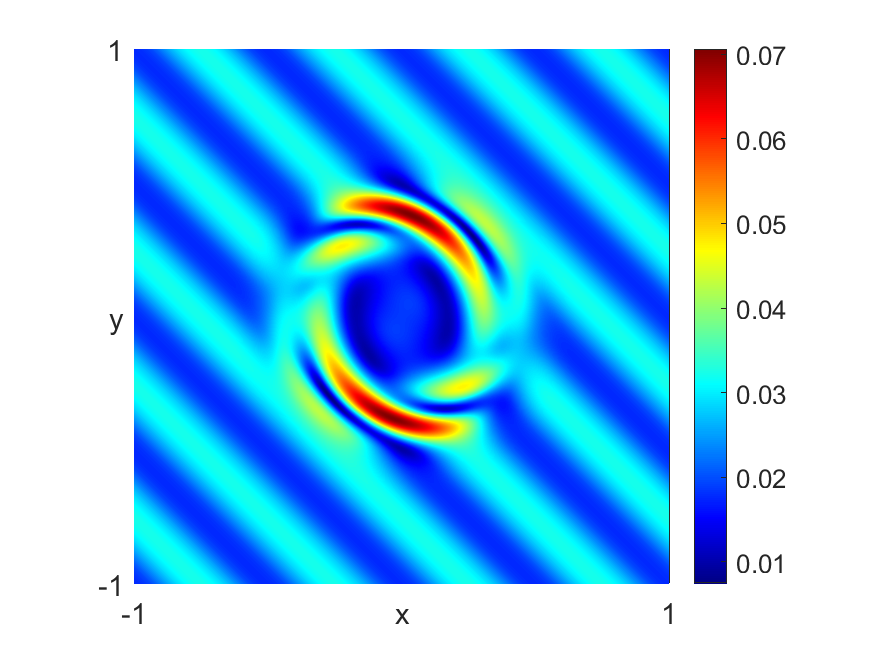}
		\caption{ \bf $\vr$ - Deviation }
	\end{subfigure}	
	\begin{subfigure}{0.32\textwidth}
		\includegraphics[width=\textwidth]{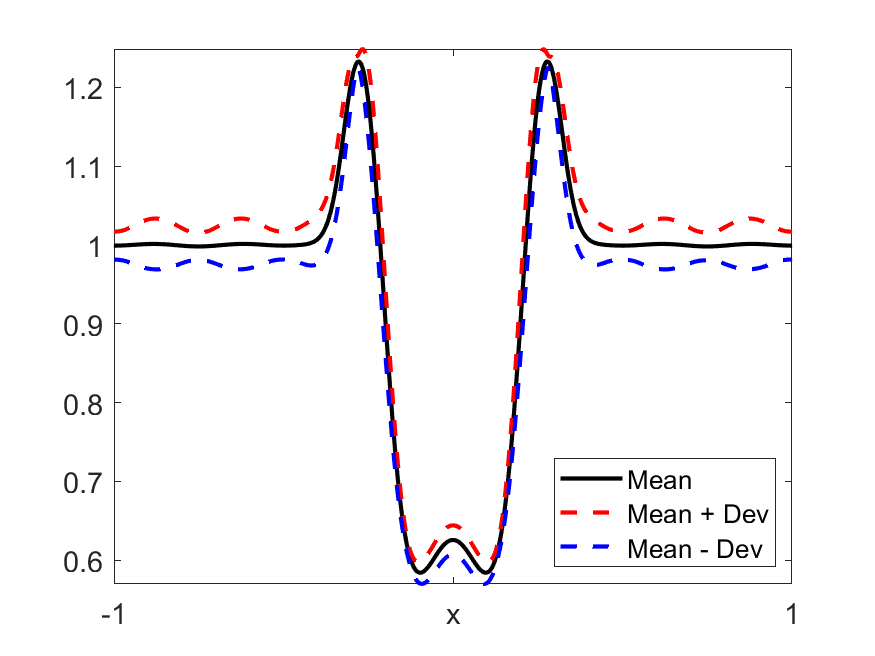}
		\caption{ \bf $\vr$ }
	\end{subfigure}	\\		
	\begin{subfigure}{0.32\textwidth}
		\includegraphics[width=\textwidth]{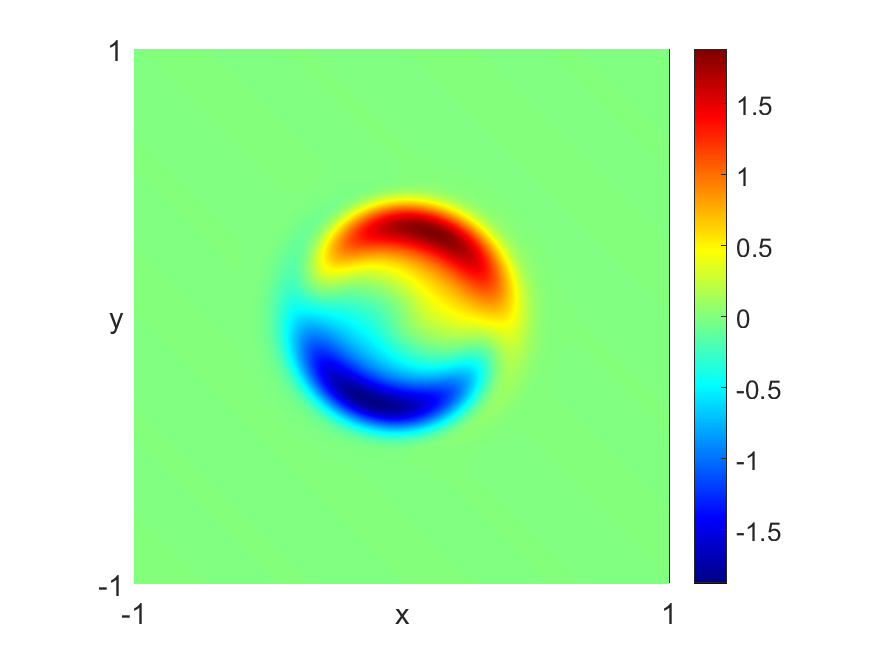}
		\caption{ \bf $m_1$ - Mean}
	\end{subfigure}	
	\begin{subfigure}{0.32\textwidth}
		\includegraphics[width=\textwidth]{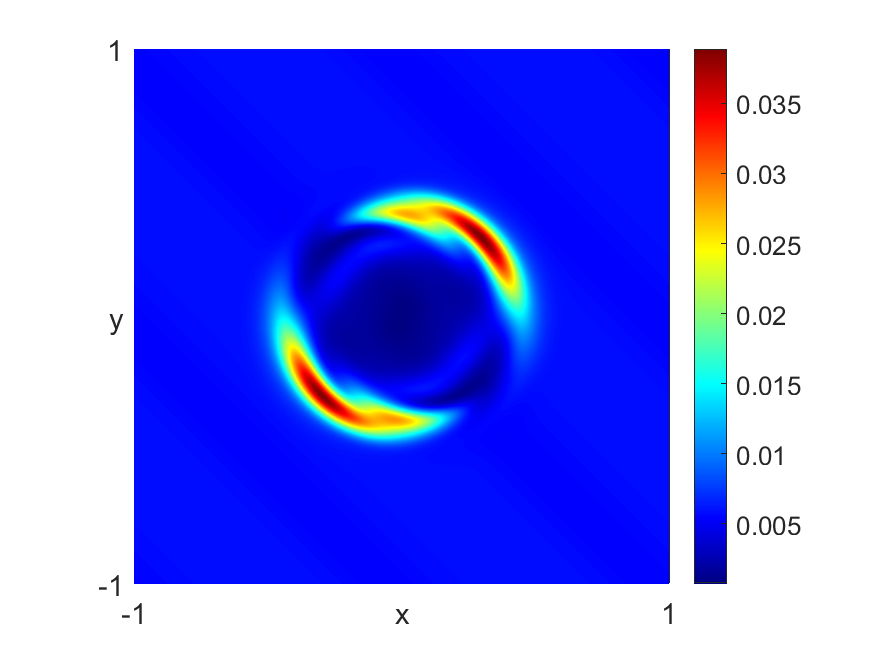}
		\caption{ \bf $m_1$ - Deviation}
	\end{subfigure}	
	\begin{subfigure}{0.32\textwidth}
		\includegraphics[width=\textwidth]{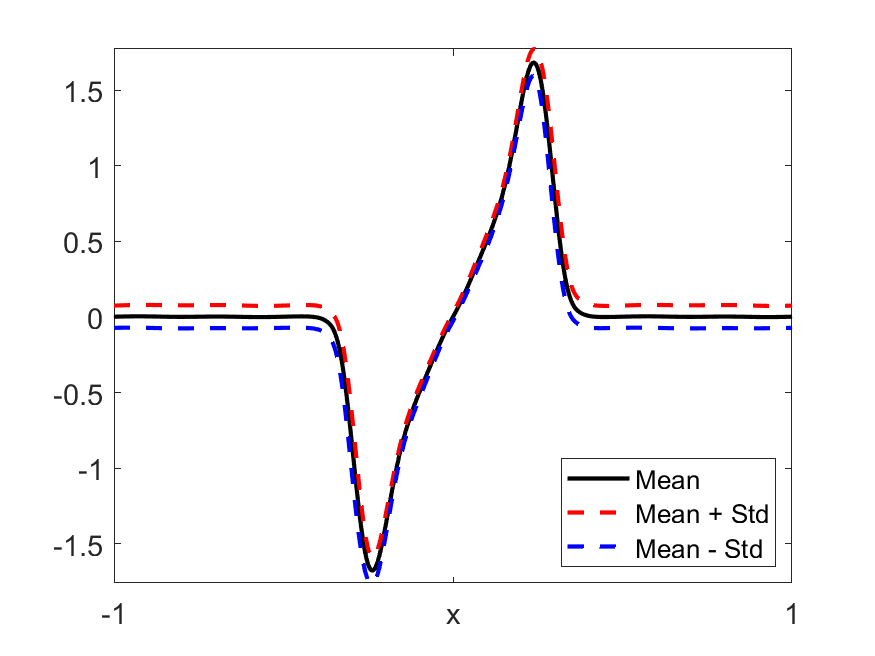}
		\caption{ \bf $m_1$}
	\end{subfigure}	\\
	\begin{subfigure}{0.32\textwidth}
		\includegraphics[width=\textwidth]{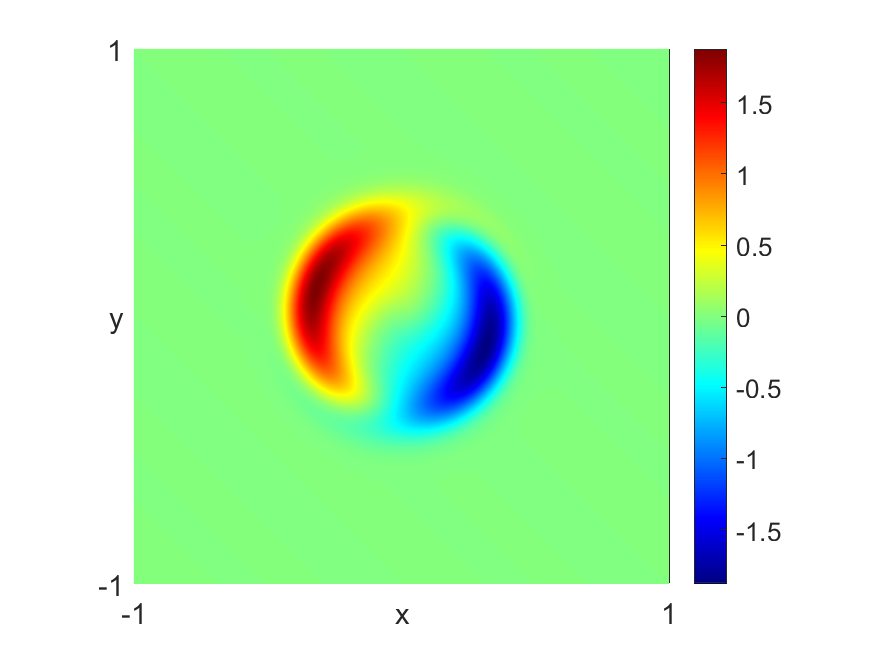}
		\caption{ \bf $m_2$ - Mean}
	\end{subfigure}	
	\begin{subfigure}{0.32\textwidth}
		\includegraphics[width=\textwidth]{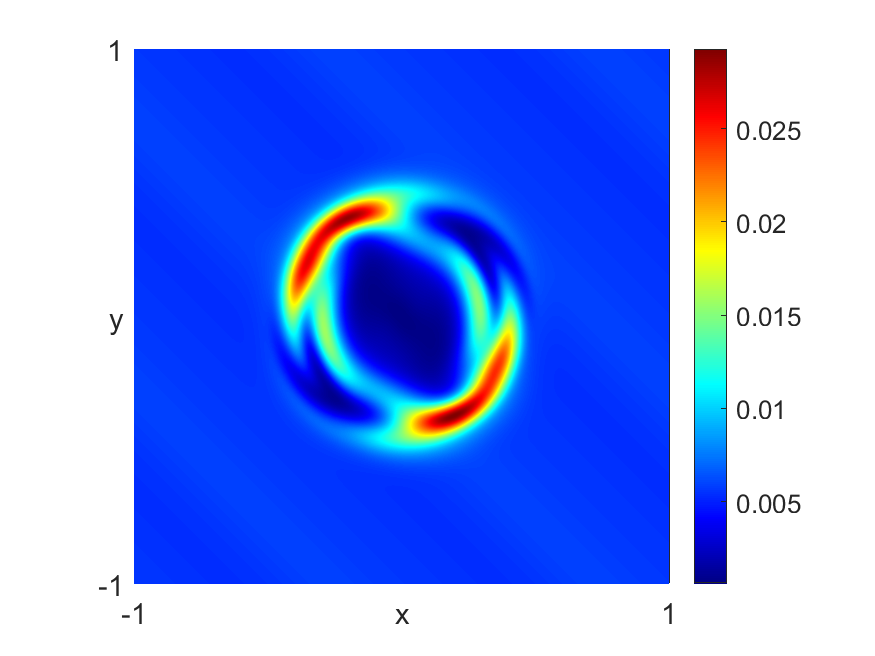}
		\caption{ \bf $m_2$ - Deviation}
	\end{subfigure}	
	\begin{subfigure}{0.32\textwidth}
		\includegraphics[width=\textwidth]{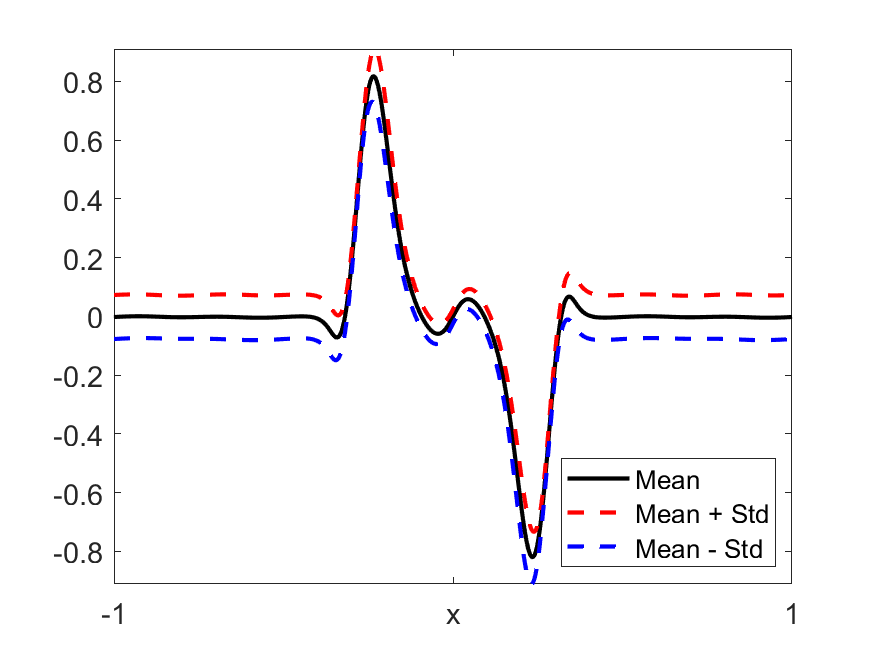}
		\caption{ \bf $m_2$}
	\end{subfigure}	\\
	\begin{subfigure}{0.32\textwidth}
		\includegraphics[width=\textwidth]{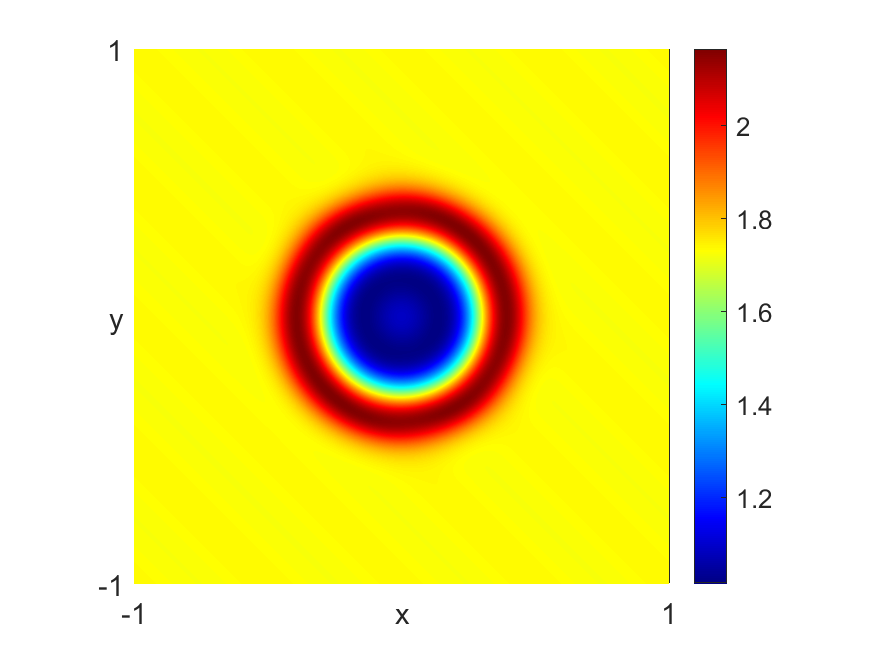}
		\caption{ \bf $S$ - Mean}
	\end{subfigure}	
	\begin{subfigure}{0.32\textwidth}
		\includegraphics[width=\textwidth]{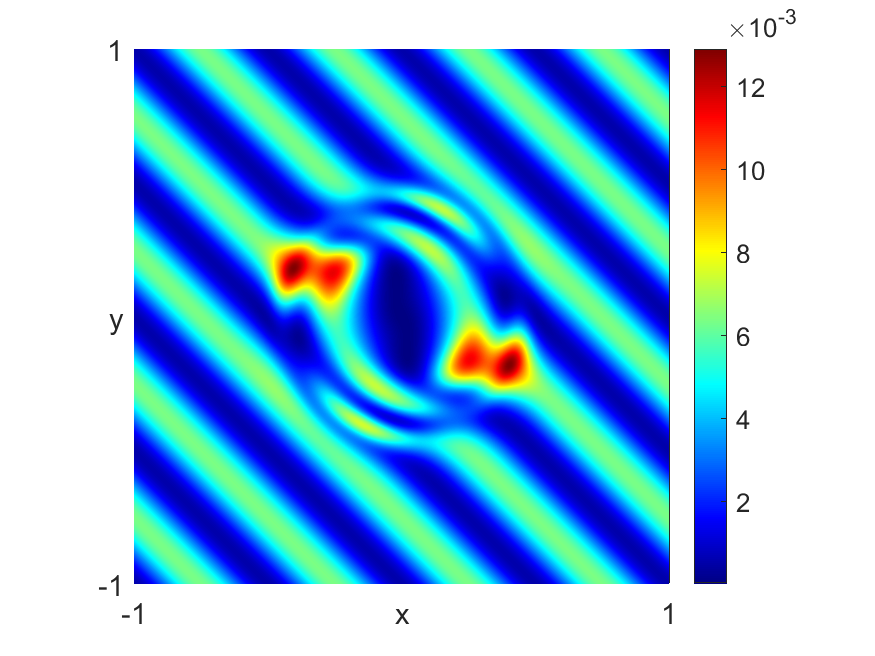}
		\caption{ \bf $S$ - Deviation}
	\end{subfigure}	
	\begin{subfigure}{0.32\textwidth}
		\includegraphics[width=\textwidth]{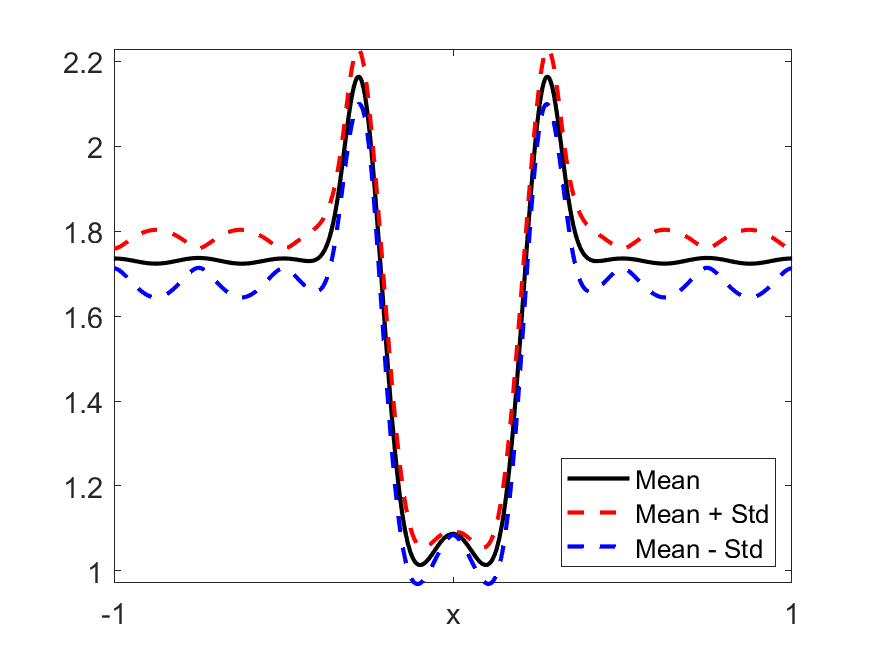}
		\caption{ \bf $S$}
	\end{subfigure}	
	\caption{  \small{Vortex Experiment: {\bf Numerical solutions} obtained by the Monte Carlo FV method on a mesh with $512^2$ cells. Left: Mean-value of $\vr, \vm, S$; Middle: Deviation of $\vr, \vm, S$; Right:  $\vr, \vm, S$ at the line $x = y$.}}\label{fig1}
\end{figure}

\begin{figure}[htbp]
	\setlength{\abovecaptionskip}{0.cm}
	\setlength{\belowcaptionskip}{-0.cm}
	\centering
	\includegraphics[width=0.49\textwidth]{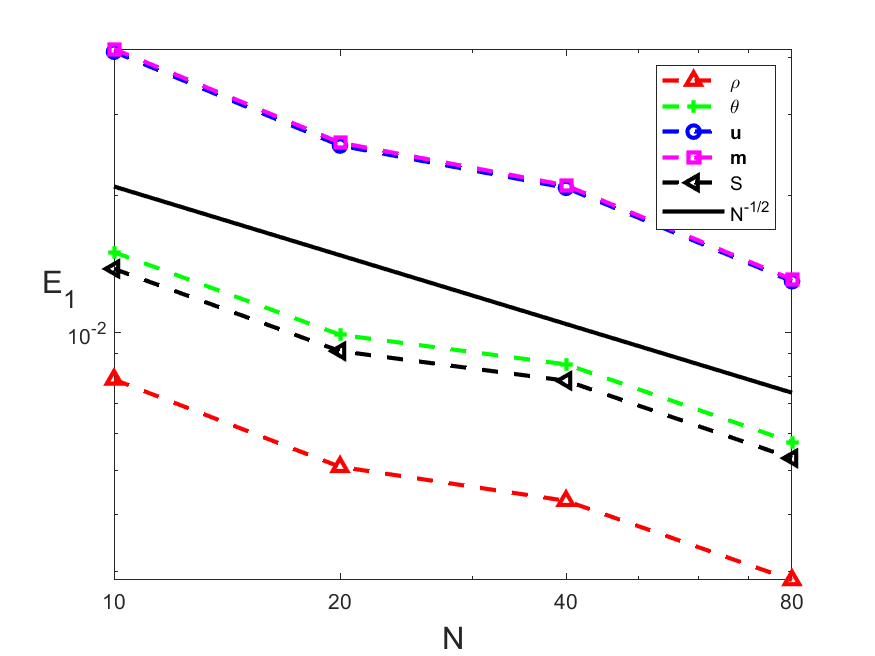}
		\includegraphics[width=0.49\textwidth]{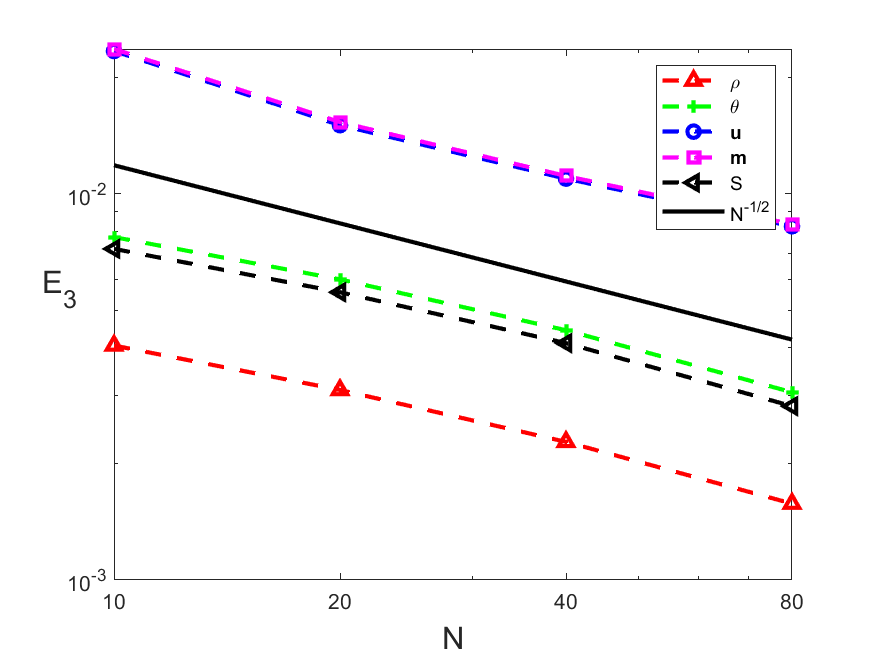}
	\
	\caption{  \small{Vortex Experiment:  {\bf Statistical errors} of the mean $E_1$, ($E_1=E_2$) and the deviation $E_3$, ($E_3=E_4$) in $L^1$-norm obtained with a fine mesh parameter $h_{ref} = 2/512$ and different $N = 10 \cdot 2^n, n = 0, \dots, 3$.
	The black solid lines without any marker denote the reference slope of $N^{-1/2}$.  }}\label{fig2}
\end{figure}

\begin{figure}[htbp]
	\setlength{\abovecaptionskip}{0.cm}
	\setlength{\belowcaptionskip}{-0.cm}
	\centering
		\includegraphics[width=0.49\textwidth]{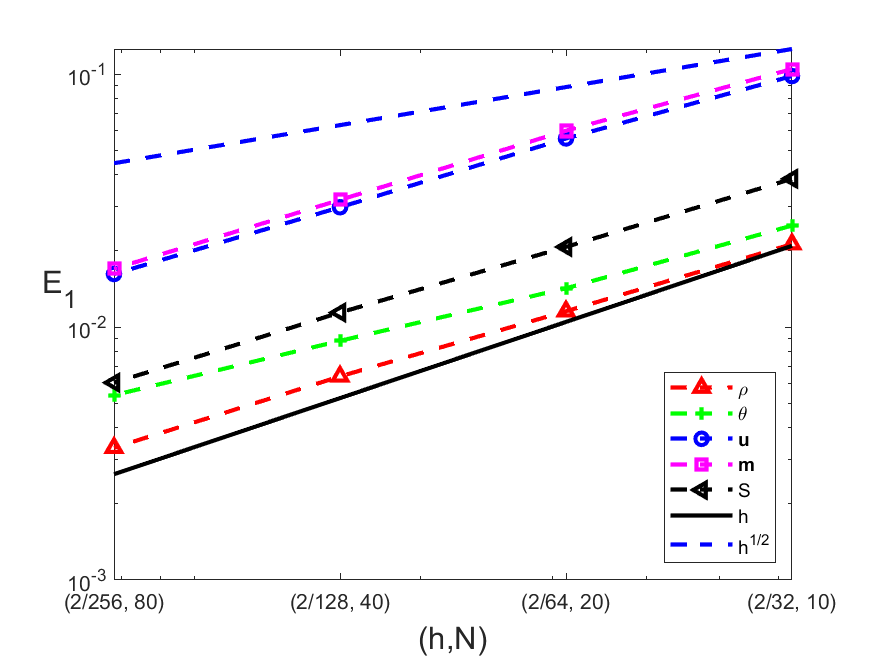}
		\includegraphics[width=0.49\textwidth]{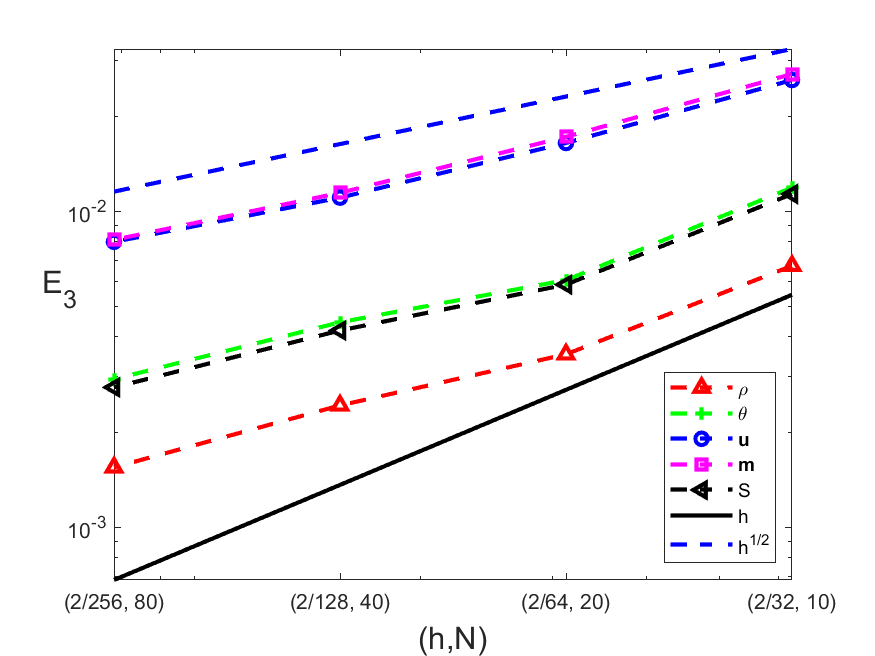}
		\
	\caption{  \small{Vortex Experiment:   {\bf Total errors} of the mean $E_1$, ($E_1=E_2$) and the deviation $E_3$, ($E_3=E_4$) in $L^1$-norm with $h = 2/(32 \cdot 2^n), N = 10 \cdot 2^n, n = 3, \dots, 0$.  The blue dashed and black solid lines without any marker denote the reference slope of $h^{1/2}$ and $h$, respectively. }}\label{fig3} 
\end{figure}

\section{Conclusion}\label{sec_con}

In 1958 J.~Nash advocated in his pioneering work \cite{Nash} that boundedness of temperature and density can be a suitable criterion to prove (conditional) existence, smoothness and uniqueness for flow equations. Indeed, this hypothesis has been recently rigorously proved
for  heat conductive viscous compressible fluids governed by the Navier--Stokes--Fourier system by Feireisl, Wen, and Zhu \cite{FeiWenZhu}, see
also Basari\'c, Feireisl, Mizerov\'a~\cite{BFM} for the case of Dirichlet boundary conditions.
%
%
These results indicate that the regularity of the Navier--Stokes--Fourier solutions is a generic property. Consequently, we assume in this paper that solutions with a possible blow-up are statistically insignificant. Under a rather weak assumption that numerical densities and temperatures are bounded in probability, cf.~\eqref{bip}, we perform
statistical analysis of the random compressible Navier--Stokes--Fourier equations. Numerical solutions are obtained by means of the Monte Carlo method coupled with a suitable structure-preserving, consistent and stable numerical method for space-time  discretisation. In particular, we choose viscous finite volume method~\eqref{scheme}, but
any consistent and stable deterministic method can be used as well.

In Theorem~\ref{FVT1} we have proved  the convergence of the Monte Carlo estimators for the expectation and deviation. The convergence proof of finite volume solutions in random space is nontrivial and requires intrinsic stochastic
compactness arguments, such as  the Skorokhod representation theorem
and the Gy\"ongy–Krylov theorem.
In Theorem~\ref{FVT2} we present the error estimates of the Monte Carlo FV approximations for the expectation and deviation.
As far as we are aware the present results are the first rigorous convergence and error analysis results for the Monte Carlo method applied to
heat conductive viscous compressible flows.
The numerical experiment presented in Section~\ref{num} illustrates theoretical results.

\section*{Acknowledgements}
The authors sincerely thank E.~Feireisl (Prague) for stimulating discussions.

\def\cprime{$'$} \def\ocirc#1{\ifmmode\setbox0=\hbox{$#1$}\dimen0=\ht0
  \advance\dimen0 by1pt\rlap{\hbox to\wd0{\hss\raise\dimen0
  \hbox{\hskip.2em$\scriptscriptstyle\circ$}\hss}}#1\else {\accent"17 #1}\fi}

\appendix

\section{Some basic statistical results}
In this section we recall some basic theories in statistical analysis from Ledoux and Talagrad \cite{LedTal}.
\begin{Lemma}[Corollary 7.10 of \cite{LedTal}]\label{LT710}
Let $X$ be a Borel random variable with values in a separable Banach space $B$. Let $(X_i)_{i\in \mathbb N}$ be a sequence of independent copies of $X$. Denote $S_N=X_1+\dots+X_N$ for $N\geq 1$.
Then
\[\frac{S_N}{N} \to 0 \mbox{ a.s. }
\]
if and only if $\Expec{ \norm{X}} <\infty$ and $\Expec{X} =0$.
Note that ``a.s." stands for ``almost surely".
\end{Lemma}

\begin{Lemma}[Proposition 9.11 of \cite{LedTal}]\label{LT911}
Let $B$ be a Banach space of type $p$ and cotype $q$ with type constant $C_1$ and cotype constant $C_2$. Then, for every finite sequence $(X_i)$of independent mean zero Radon random variables in $L_p(B)$ (resp. $L_q(B)$),
\[ \Expec{ \norm{\sum_i X_i}^p} \leq (2C_1)^p \sum_i \Expec{\norm{ X_i}^p} \mbox{ and }
\Expec{\norm{\sum_i X_i}^q} \geq (2C_2)^{-q} \sum_i \Expec{\norm{ X_i}^q}.
\]
\end{Lemma}


\begin{Lemma}[Theorem 10.5 of \cite{LedTal}]\label{LT105}
Let $X$ be a mean zero random variable such that
$\Expec{\norm{X}^2} < \infty$
  with values in a separable Banach space $B$ of type 2. Then $X$ satisfies the central limit theorem. Conversely, if in a (separable) Banach space $B$, every random variable $X$ such that $\Expec{ X} =0$ and $\Expec{\norm{X}^2} < \infty$  satisfies the central limit theorem, then $B$ must be of type 2.
\end{Lemma}

\section{Finite volume method}
\label{FV}
For completeness we present the finite volume method proposed in \cite{FLMS_FVNSF}. Let the physical domain $\Td$ be divided into structured control volumes (cuboids for simplicity)
$\Td  = \bigcup_{K \in \grid_h} K.$
The finite volume approximation $(\vr_{h}, \vu_{h}, \vth)$ is defined as a piecewise constant in space and time function that solves the following nonlinear algebraic system:
\begin{subequations}\label{scheme}
\setlength\belowdisplayskip{0pt}
 \setlength\abovedisplayskip{0pt}
\begin{equation}\label{scheme_D}
\intTd{ D_t \vrh  \phi_h } - \intfacesint{  \Fup (\vrh ,\vuh ) \jump{\phi_h}   } = 0 \quad \mbox{for all}\ \phi_h \in Q_h,
\end{equation}
\begin{multline}
\label{scheme_M}
\intTd{ D_t  (\vrh  \vuh ) \cdot \bfphi_h } - \intfacesint{ \Fup  (\vrh  \vuh ,\vuh ) \cdot \jump{\bfphi_h}   }
\\ +\intTd{ (\bS_h - p_h \I):\Gradh \bfphi_h } = \intTd{\vrh  \vc{g} \cdot \bfphi_h}
\quad \mbox{for all } \bfphi_h \in \vQh,
\end{multline}
\begin{multline} \label{scheme_T}
c_v\intTd{ D_t (\vrh  \vth ) \phi_h } - c_v\intfacesint{  \Fup (\vrh \vth ,\vuh )\jump{\phi_h} }
+\intfacesint{  \frac{\kappa}{ h } \jump{\vth}  \jump{ \phi_h}  }
\\
= \intTd{  (\bS_h-p_h \I ):\Gradh \vuh  \phi_h}, \quad \mbox{for all}\ \phi_h \in Q_h ,
\end{multline}
\end{subequations}
where $Q_h$ is the space of piecewise constant functions on $\grid_h$, $\vQh=Q_h^d$,
\begin{align*}
p_h(t)&= \vrh(t) \vth(t)  \quad \mbox{ for } \vth(t) >0  \quad  \mbox{ and } \quad p_h(t) = 0 \quad \mbox{ for } \vth(t) \leq 0,
\\
\Fup (r_h,\vuh)
&=\Up[r_h, \vuh] - \muh \jump{ r_h }, \;\eps \in (-1,1),
\\
\Up [r_h, \vuh]
&= \avs{r_h} \ \avs{\vuh} \cdot \vn - \frac{1}{2} |\avs{\vuh} \cdot \vn| \jump{r_h},
\\
\bS_h &= 2 \mu \Dhuh  + \lambda  \Divh   \vuh \I , \quad \Dhuh = (\Gradh \vu_h+\Gradh^T \vu_h)/2,
\\ \Gradh r_h & = \sum_{K \in \grid_h}  (\Gradh r_h)_K 1_K,  \quad
(\Gradh r_h)_K = \sum_{\sigma\in \pd K} \frac{|\sigma|}{|K|} \avs{r_h} \vn \quad \mbox{for } r_h \in Q_h,
\\
\Divh \vvh  &= \sum_{K \in \grid_h}  (\Divh  \vvh)_K 1_K, \quad
(\Divh \vvh)_K = \sum_{\sigma\in \pd K} \frac{|\sigma|}{|K|} \avs{\vvh} \cdot \vn
\quad \mbox{for } \vv_h \in \vQh,
\\
 \Pi_h v &= \sum_{K \in \grid_h} \mathds{1}_K \frac{1}{|K|} \int_K v \dx \quad \mbox{for } v\in L^1(\Td),
 \\
  D_t r_h (t) &= \frac{r_h(t) - r_h(t-\Delta t)}{\Delta t},\quad
\avs{r_h} = \frac{r_h^{\rm in} + r_h^{\rm out} }{2},\ \ \
\jump{r_h }  = r_h^{\rm out} - r_h^{\rm in}
\end{align*}
and $r_h^{\rm out}$ and $r_h^{\rm in}$ are respectively the outward and inward limits with respect to a given normal $\vn$ to an interface of the control volumes.

\section{Figure supplements of $L^2$-errors}
\begin{figure}[htbp]
	\setlength{\abovecaptionskip}{0.cm}
	\setlength{\belowcaptionskip}{-0.cm}
	\centering
	\includegraphics[width=0.49\textwidth]{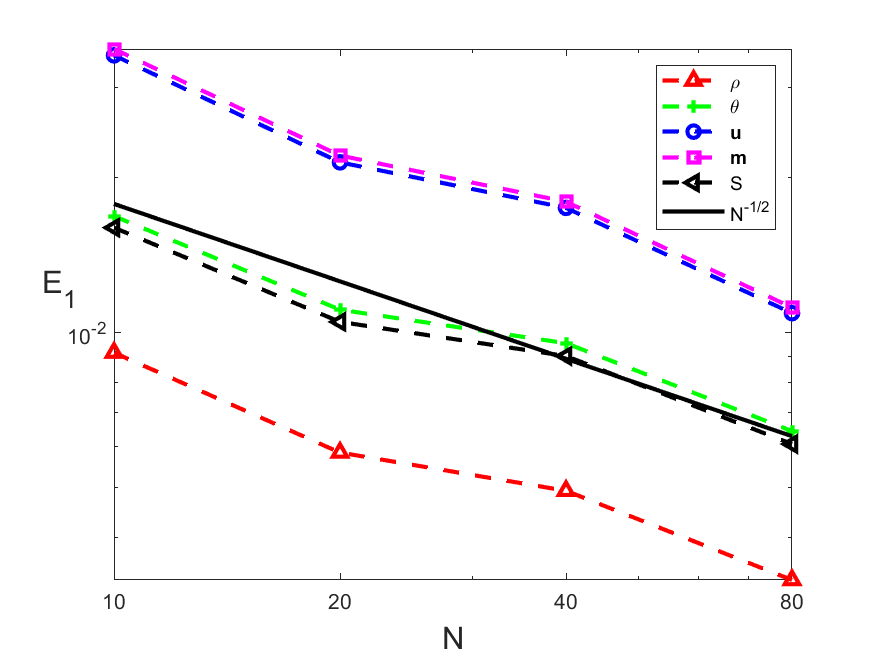}
		\includegraphics[width=0.49\textwidth]{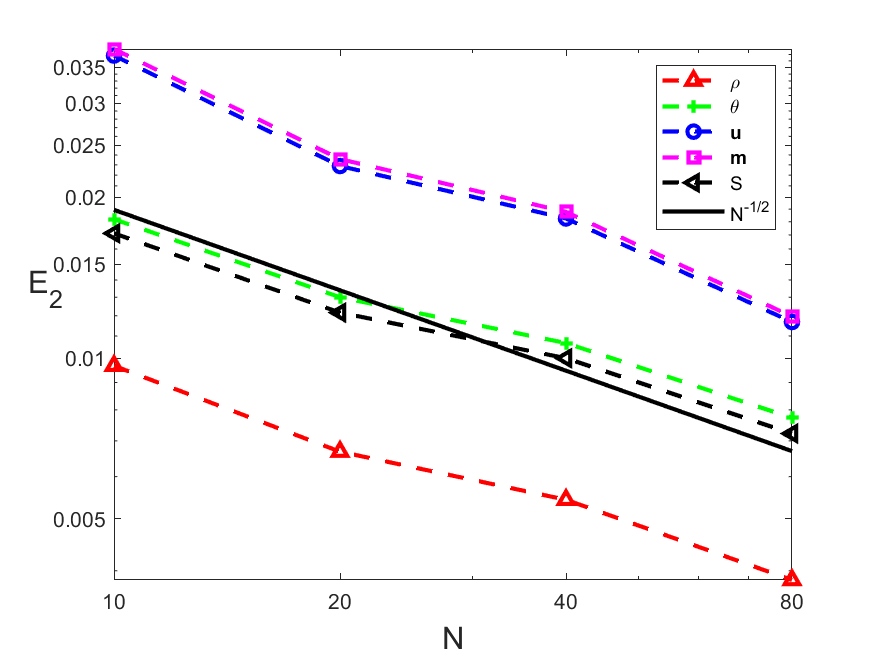}\\
		\includegraphics[width=0.49\textwidth]{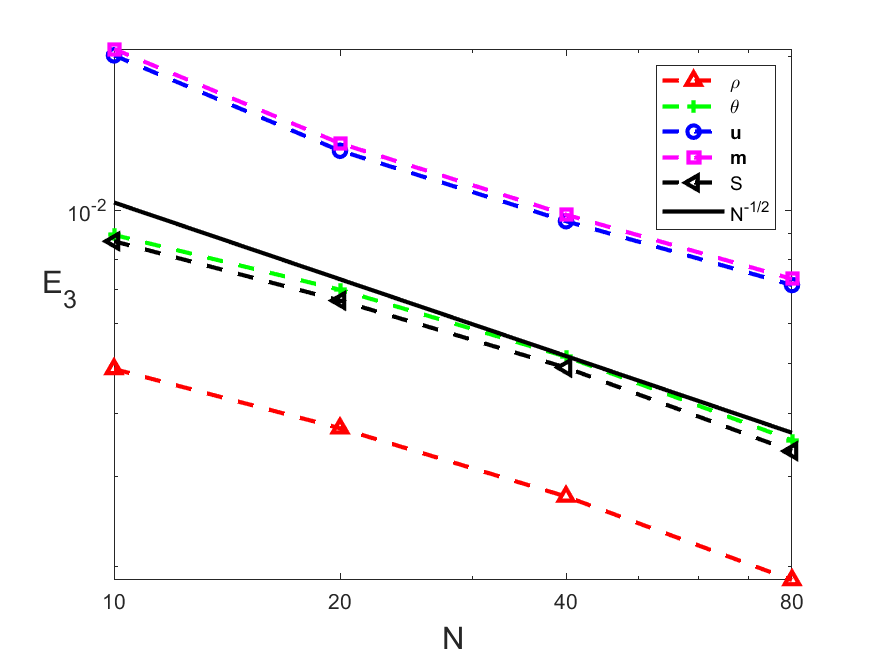}
		\includegraphics[width=0.49\textwidth]{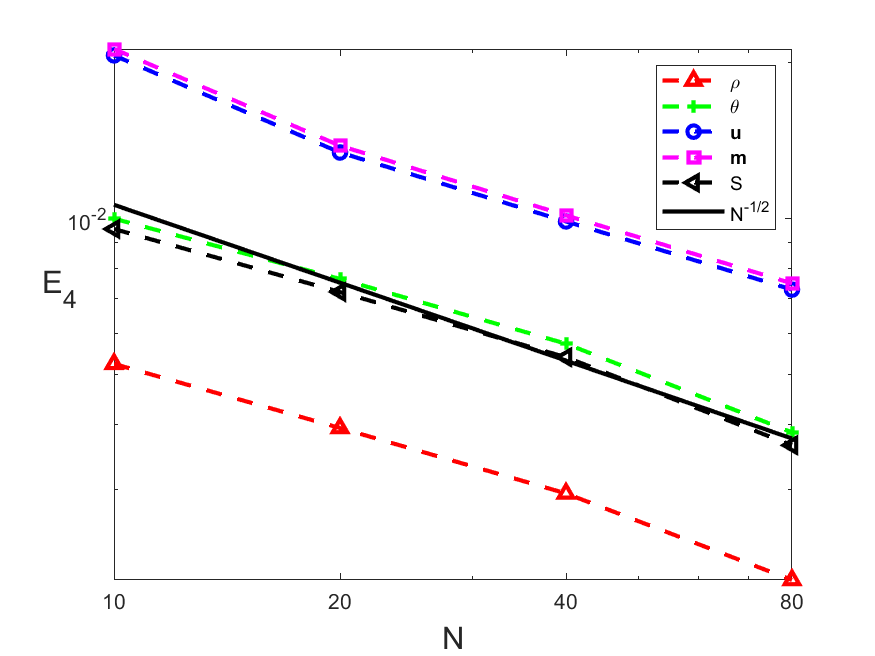}\\	
	\
	\caption{  \small{Vortex Experiment:  {\bf Statistical errors} of the mean ($E_1, E_2$) and the deviation ($E_3, E_4$) in $L^2$-norm obtained with a fine mesh parameter $h_{ref} = 2/512$ and different $N = 10 \cdot 2^n, n = 0, \dots, 3$.
	The black solid lines without any marker denote the reference line of $N^{-1/2}$.  }}\label{fig4}
\end{figure}

\begin{figure}[htbp]
	\setlength{\abovecaptionskip}{0.cm}
	\setlength{\belowcaptionskip}{-0.cm}
	\centering
		\includegraphics[width=0.49\textwidth]{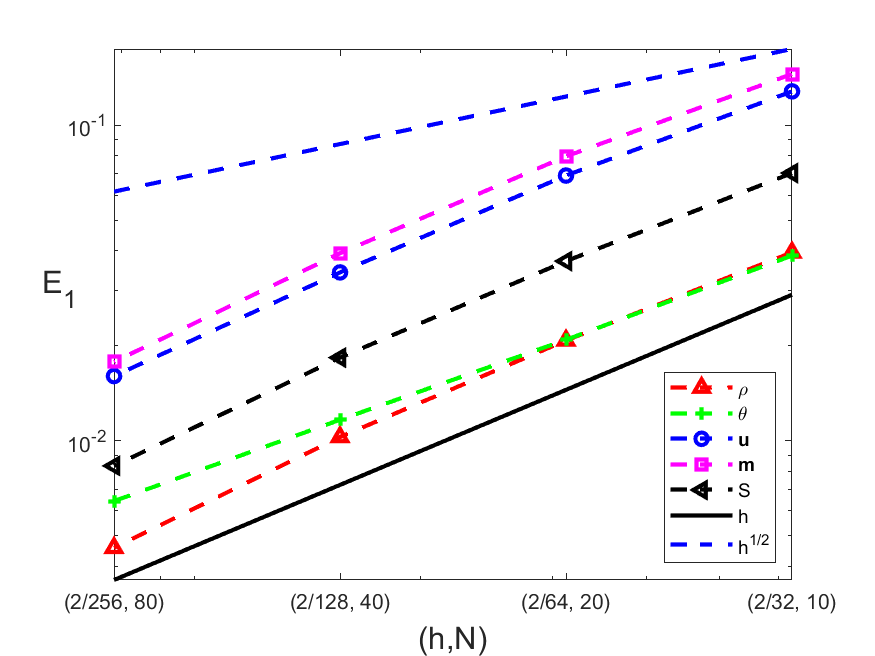}
		\includegraphics[width=0.49\textwidth]{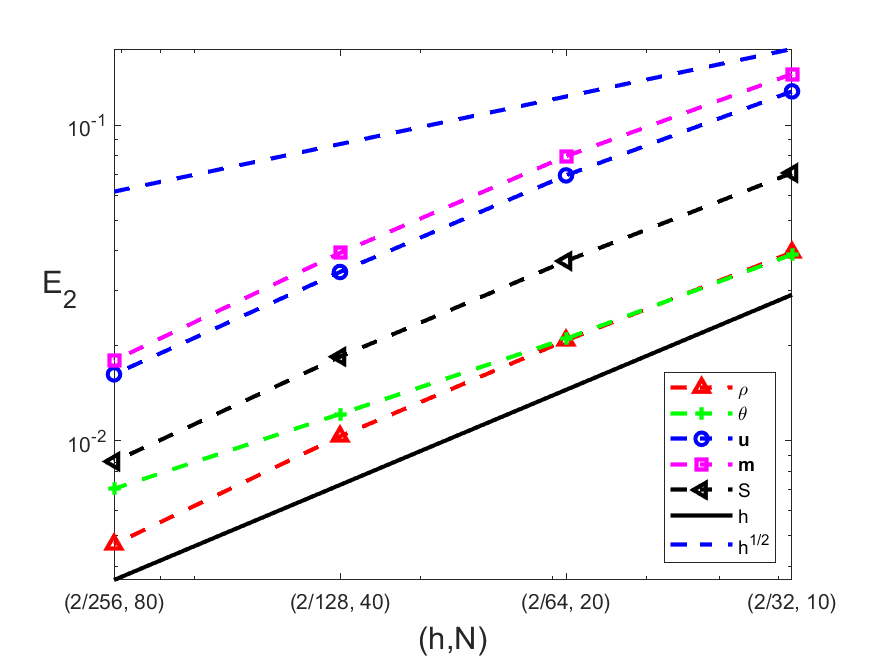}\\
		\includegraphics[width=0.49\textwidth]{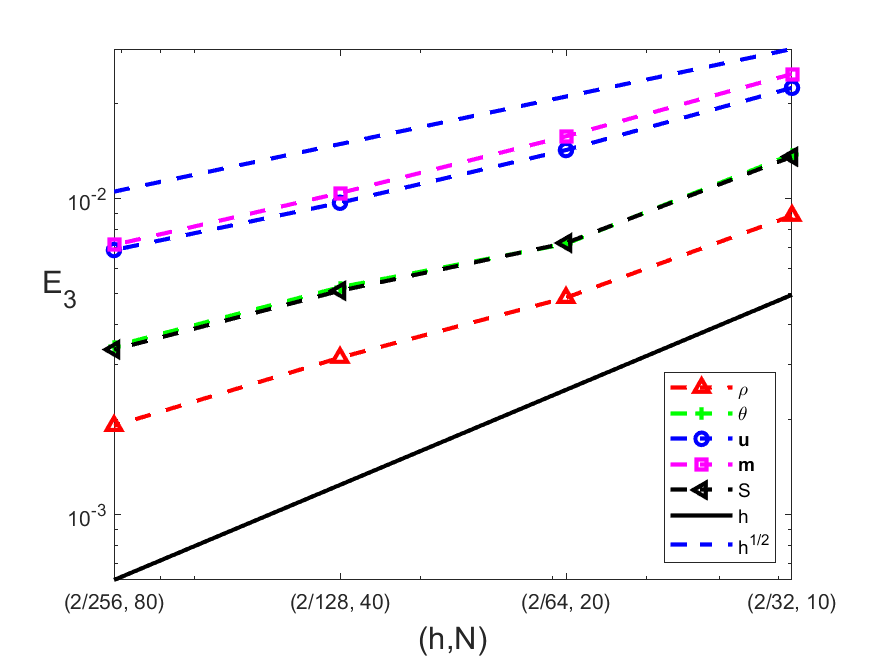}
		\includegraphics[width=0.49\textwidth]{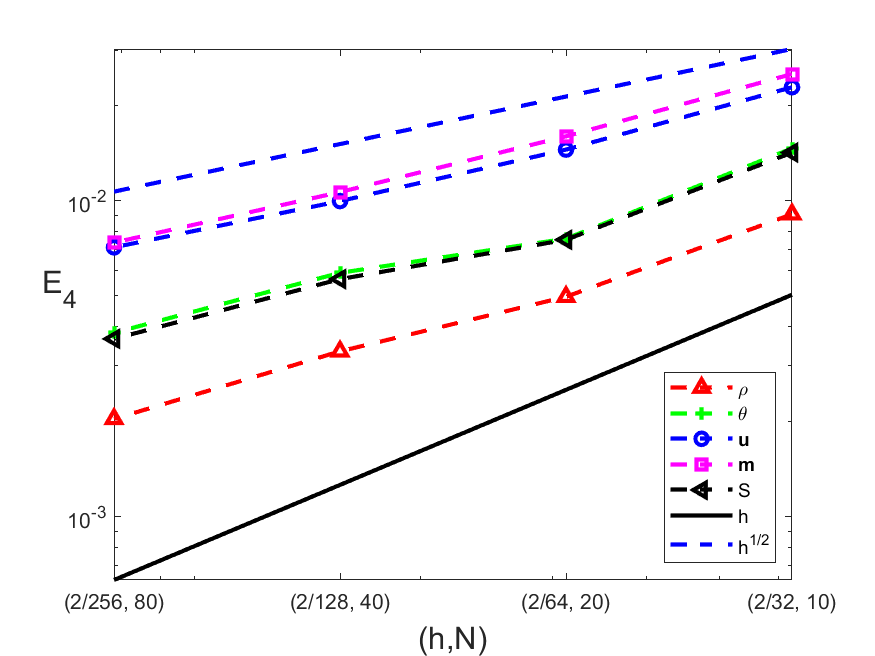}\\
		\
	\caption{  \small{Vortex Experiment:   {\bf Total errors} of the mean ($E_1, E_2$) and the deviation ($E_3, E_4$) in $L^2$-norm with $h = 2/(32 \cdot 2^n), N = 10 \cdot 2^n, n = 3, \dots, 0$.  The blue dashed and black solid lines without any marker denote the reference slope of $h^{1/2}$ and $h$, respectively. }}\label{fig5}
\end{figure}

\end{document}